\documentclass[a4paper,11pt,english]{amsart}
\usepackage[applemac]{inputenc}
\usepackage{bbm}
\usepackage[english]{babel}
\usepackage{amsfonts}
\usepackage{amsmath} 
\usepackage{amsthm}
\usepackage{hyperref} 
\usepackage{latexsym}
\usepackage{mathrsfs}
\usepackage{array}
\usepackage{amssymb}
\usepackage{enumerate}
\usepackage{graphicx}
\usepackage{color}
\usepackage{stmaryrd}
\newcommand{\ph}{\phantomsection}

\newcommand{\R}{\mathbb  R}
\newcommand{\C}{\mathbb  C}
\newcommand{\N}{\mathbb  N}

\newcommand{\eps}{\varepsilon}
\renewcommand{\epsilon}{\varepsilon}
\newcommand{\e}{  \text{e}   }

\newcommand{\Z}{  \mathbb{Z}   }
\newcommand{\F}{  \mathcal{F}   }

\renewcommand{\H}{  \mathcal{H}   }

\newcommand{\T}{  \mathcal{T} }

\newcommand{\dis}{\displaystyle}

\newcommand{\ov}{  \overline  }

\renewcommand{\phi}{  \varphi  }

\newcommand{\wh}{  \widehat   }
\newcommand{\<}{  \langle   }

\renewcommand{\>}{  \rangle   }

\renewcommand{\S}{  \mathbb{S}  }

\numberwithin{equation}{section}
\theoremstyle{plain}

 \newtheorem{theo}{Theorem}[section]
\newtheorem{lemm}[theo]{Lemma}
\newtheorem{prop}[theo]{Proposition}
\newtheorem{coro}[theo]{Corollary}

\newtheorem{rema}[theo]{Remark}

\def\beq{\begin{equation}}   \def\eeq{\end{equation}}
\def\bea{\begin{eqnarray}}  \def\eea{\end{eqnarray}}

\renewcommand{\theequation}{\thesection.\arabic{equation}}
\newcounter{hran} \renewcommand{\thehran}{\thesection.\arabic{hran}}

\def\bmini{\setcounter{hran}{\value{equation}}
    \refstepcounter{hran}\setcounter{equation}{0}
    \renewcommand{\theequation}{\thehran\alph{equation}}\begin{eqnarray}}

\def\bminiG#1{\setcounter{hran}{\value{equation}}
\refstepcounter{hran}\setcounter{equation}{-1}
\renewcommand{\theequation}{\thehran\alph{equation}}
\refstepcounter{equation}\label{#1}\begin{eqnarray}}

\author{Pierre Germain}
\address{Courant Institute of Mathematical Sciences, 251 Mercer Street, New York 10012-1185 NY, USA}
\email{pgermain@cims.nyu.edu}
\author{Zaher Hani}
\address{School of Mathematics, Georgia Institute of Technology, Atlanta, GA 30332, USA}
\email{hani@math.gatech.edu}

\author{ Laurent Thomann }
\address{Laboratoire de Math\'ematiques J. Leray, UMR  6629 du CNRS, Universit\'e de Nantes, 
2, rue de la Houssini\`ere,
44322 Nantes Cedex 03, France}
\email{laurent.thomann@univ-nantes.fr}

\title[ ]{ On the continuous resonant equation for NLS\\
I. Deterministic analysis}

\begin{document}

\begin{abstract}
We study the continuous resonant (CR) equation which was derived in \cite{FGH} as the large-box limit of the cubic nonlinear Schr\"odinger equation in the small nonlinearity (or small data) regime. We first show that the system arises in another natural way, as it also corresponds to the resonant cubic Hermite-Schr\"odinger equation (NLS with harmonic trapping). We then establish that the basis of special Hermite functions is well suited to its analysis, and uncover more of the striking structure of the equation. We study in particular the dynamics on a few invariant subspaces: eigenspaces of the harmonic oscillator, of the rotation operator, and the Bargmann-Fock space. We focus on stationary waves and their stability.
\end{abstract}

\thanks{P. G. is partially supported by NSF grant DMS-1101269, a start-up grant from the Courant Institute, and a Sloan fellowship.}
\thanks{L.T. is partially supported   by the  grant  ``ANA\'E'' ANR-13-BS01-0010-03}
\thanks{Z.~H. is partially supported by NSF Grant DMS-1301647, and a start-up fund
from the Georgia Institute of Technology.}
\maketitle

\section{Introduction}

\subsection{Presentation of the equation}\label{intro}
The purpose of this manuscript is to study the so-called {\it continuous resonant equation} which was introduced by Faou-Germain-Hani \cite{FGH} as the large-box limit of the cubic nonlinear Schr\"odinger equation in the small nonlinearity regime. This equation reads
\begin{equation}\label{CR}\tag{CR}
\left\{
\begin{aligned}
&i\partial_{t}u=\T(u,u,u), \quad   (t,x)\in \R\times \R^2,\\
&u(0,x)=  f(x),
\end{aligned}
\right.
\end{equation}
where  the nonlinearity is defined by 
\begin{align}\label{lim}
\T(f_1,f_2,f_3)(z) & \overset{def}{=}  \int_\mathbb{R} \int_{\R^2}  f_1(x+z)f_2(\lambda x^{\perp}+z)\ov{f_3(x+\lambda x^{\perp}+z)} \,dx\, d\lambda \\
& =  \int_\mathbb{R} \int_{\R^2}  f_1(x^{\perp}+z)f_2(\lambda x+z)\ov{f_3(x^{\perp}+\lambda x+z)} \,dx\, d\lambda,\nonumber
\end{align}
for any $z\in \R^2$ (if $x=(x_1,x_2)$ we set $x^\perp=(-x_2,x_1)$).  While the above formula seems mysterious at this stage, it can be thought of as an integration over all rectangles for which $z$ is a vertex. Indeed, the points $z$, $x + z$, $\lambda x^\perp +z$, $x + \lambda x^\perp + z$ form a rectangle in $\mathbb{R}^2$, and this yields a parameterization of all rectangles which count $z$ as a vertex\footnote{This is related to the well-known fact that four frequencies $\xi_1$, $\xi_2$, $\xi_3$, $\xi_4$ are resonant for $NLS$ (say, on the 2-torus), if $\xi_1 + \xi_2 = \xi_3 + \xi_4$ and $|\xi_1|^2 + |\xi_2|^2 = |\xi_3|^2 + |\xi_4|^2$, which is equivalent to these frequencies forming a rectangle.}. \medskip

This expression can also be reformulated using the unitary group $e^{it\Delta}$, as was observed in \cite{FGH}. Note that the definition of $\T$ above is slightly different but equivalent to that in\;\cite{FGH} as we explain in Section\;\ref{known results}. Here we will show that  $\T$ can also be reformulated using the semigroup  $e^{it(-\Delta+|x|^2)}$ (see  Lemma\;\ref{lem11} below). The key ingredient is the so-called lens transform which is an explicit formula (given in \eqref{lenstransf}) which links   $e^{-it\Delta}$ to $e^{it(-\Delta+|x|^2)}$. This suggests that the harmonic oscillator ${H=-\Delta+|x|^2}$ will play a central role in the study of \eqref{CR}.
\medskip

Defining
\begin{align*}
\mathcal E(f_1,f_2,f_3,f_4) & \overset{def}{=} \langle \mathcal{T}(f_1,f_2,f_3)\,,\,f_4 \rangle_{L^2}\\
& = \int_\mathbb{R} \int_{\R^2} \int_{\R^2} f_1(x+z)f_2(\lambda x^{\perp}+z)\ov{f_3(x+\lambda x^{\perp}+z)} \, \ov{f_4( z) }\,dz\, dx\,d\lambda, \nonumber
\end{align*}
it is easy to check that the \eqref{CR} equation derives from the Hamiltonian 
$$
\mathcal{E}(f) \overset{def}{=} \mathcal{E}(f,f,f,f)
$$ 
given the symplectic form  $\omega(f,g) = -4\mathfrak{Im} \langle f\,,\, g \rangle_{L^2(\mathbb{R}^2)}$ on $L^2(\mathbb{R}^2)$ (this follows easily from the symmetries of $\mathcal{E}$). In other words, \eqref{CR} can also be written
$$
i\partial_t f = \frac{1}{2} \frac{\partial \mathcal{E}(f)}{\partial \bar f}.
$$
Important quantities conserved by the flow of the above equation (we shall come back to them) are the mass $M$ and angular momentum $P$:
$$
M \overset{def}{=} \int_{\mathbb{R}^2}|u|^2 \qquad \mbox{and} \qquad P  \overset{def}{=} \int_{\mathbb{R}^2} i (x \times \nabla) u\, \overline u,
$$
where $x \times \nabla = x_2 \partial_{x_1} - x_1 \partial_{x_2}$.

\subsection{Physical and mathematical relevance}
The \eqref{CR} equation has rich dynamics and can be studied in its own right, but it also plays a role in the description of the dynamics of the usual cubic NLS -- with or without potential -- in various situations, which we summarize here:\vspace{3pt}
\begin{itemize}
\item[$\bullet$] It was derived in \cite{FGH} as a weakly nonlinear, big box limit of the cubic NLS 
$$
i \partial_t u - \Delta u = |u|^2 u
$$ 
(here, we consider the focusing case, but the defocusing case leads to the same picture)
on a periodic box of size $L$; equivalently, it appears as the limiting equation for high frequency envelopes of solutions of NLS on the unit torus $\mathbb T^2$. To be more specific, setting the above equation on the 2-dimensional torus of size $L$, and prescribing data of size $\epsilon$, it is  well-approximated by (CR) on very long time scales (much longer than $L^2 / \epsilon^2$).\vspace{3pt}
\item[$\bullet$] We will prove in the present paper that~\eqref{CR} can also be derived as a small data approximation of the 2-dimensional cubic NLS equation with harmonic trapping, a.k.a. Hermite-Schr\"odinger equation
\begin{equation}
\label{nls}
i \partial_t u - \Delta u + |x|^2 u = |u|^2 u
\end{equation} 
This model is widely used in several areas of physics from nonlinear optics to Bose-Einstein condensates \cite{KFC}, but its relation to the dynamics of \eqref{CR} seems to be new. To be more specific: set $H = -\Delta + |x|^2$, $\Pi_n$ the projector on the $n$-th eigenspace of $H$, and $f = e^{-itH}u$. Keeping only the totally resonant part of the nonlinearity in \eqref{nls} gives the equation
$$
i \partial_t f =  \sum_{\substack {n_1,n_2,n_3,n_4\geq 0\\n_1+n_2=n_3+n_4}}  \Pi_{n_4}\Big( (\Pi_{n_1}f_1)   (\Pi_{n_2}f_2) (\ov{ \Pi_{n_3}f_3}) \Big).
$$
We will prove that the above right-hand side is, up to a multiplicative constant, equal to $\mathcal{T}(f,f,f)$. Thus the totally resonant part of NLS with harmonic trapping is identical to~\eqref{CR}. This implies that \eqref{CR} approximates the dynamics of \eqref{nls} for large times in the small data regime, as we illustrate in Theorem \ref{approximation theorem}.\vspace{3pt}
\item The equation \eqref{CR} appears as a modified scattering limit of the cubic NLS on $\R^3$ with harmonic trapping in two directions. Therefore any information on the asymptotic  dynamics of \eqref{CR} directly gives the corresponding behaviour for NLS. We refer to Hani-Thomann \cite{HT} for more details and concrete applications.\vspace{3pt}
\item  When restricted to the Bargmann-Fock space, which is given by $L^2(\R^2)$ functions which can be written as the product of a holomorphic function with the Gaussian $e^{-\frac{|z|^2}{2}}$,  the equation\;\eqref{CR} coincides  with the model known as Lowest Landau Level~\cite{ABN, AS, Nier}, used in the  description of  rotating Bose-Einstein condensates.
\end{itemize}

The  \eqref{CR} equation can be considered as a model of NLS-like equation without dispersion.  A lot of attention has recently been paid to such equations, at least starting with the work of Colliander, Keel, Staffilani, Takaoka, and Tao \cite{CKSTT} and subsequent works on the growth of Sobolev norms \cite{Hani,HPTV}. There, the dynamics of resonant systems like (CR) arise either as approximating or asymptotic dynamics for the original NLS model (see also \cite{HT}). Another important instance of zero-dispersion Hamiltonian equations comes from the work of G\'erard and Grellier \cite{GG1,GGX} who studied the so-called cubic Szeg\"o equation on $\S^1$ (this latter equation is obtained by replacing  $\Pi_n$ in $\T$ (see Lemma \ref{lemproj}) by the projection on $e^{inx}$ for $x\in \S^1$). We refer to Pocovnicu \cite{Poco1,Poco2} for the study of Szeg\"o on $\R$. 

Despite the absence of a linear part, and of the corresponding dispersive effects, \eqref{CR} is well-posed in $L^2(\R^2)$ (see~\cite{FGH}) which is remarkable for a zero-dispersion equation with a zero-order trilinear nonlinearity (in comparison, the Szeg\"o equation is well-posed in $H^s$ for $s\geq 1/2$ and the result is sharp). In fact, the nonlinearity has a hidden smoothing property coming from Strichartz estimates of the original NLS model, and this compensates the lack of linear dispersion (see Lemma \ref{lem11}). In this direction, let us also mention the dispersion managed Schr\"odinger equation, which is obtained by averaging  a nonlinear Schr\"odinger equation with varying dispersion.  This latter equation has a similar structure as \eqref{CR}. We refer to~\cite{Z,HL} and references therein for more details.

\medskip

In \cite{GHTstoc} we undertake the study of \eqref{CR} with random initial conditions. We exhibit global rough dynamics (for initial data less regular than $L^2$), and we construct Gibbs measures which are invariant by this flow. This is related to weak turbulence theory, which is often considered to occur in the statistical regime where the phases of the Fourier coefficients are initially uncorrelated.

\subsection{Known results}\label{known results}

Here we recall some important properties of the \eqref{CR} and the operator $\mathcal T$, which were proved in~\cite{FGH}. We first clarify the slightly different formulation of the operator $\mathcal T$ we use here and that derived in \cite{FGH}. In addition to reversing the roles of $f_2$ and $f_3$, the main difference between the two definitions is that the $\lambda$ integral in \cite{FGH} is over $[-1,1]$ instead of $\R$. Denoting by $\mathcal T_{\lambda \in [-1,1]}$ the same operator as $\mathcal T$ with the $\lambda$ integral taken over $[-1,1]$ instead, one can easily show by a change of variable that:
$$
\mathcal T_{\lambda \in [-1,1]}(f_1, f_1, f_3)=\frac{1}{2}\mathcal T(f_1, f_1, f_3).
$$
In particular, one can use either formulation to define the \eqref{CR} equation in which $f_1=f_2=f_3$.

\medskip

Now we recall that the operator $\mathcal{T}$ is bounded from $L^2 \times L^2 \times L^2$ to $L^2$, and also from $\dot L^{\infty,1} \times \dot L^{\infty,1} \times \dot L^{\infty,1}$ to $ \dot L^{\infty,1}$, where $\dot L^{\infty,1}$ is given by the norm $\|f\|_{\dot L^{\infty,1}} = \| |x| f \|_{L^\infty}$.

This implies immediately that  \eqref{CR} is locally well-posed for data in $L^2$ or $\dot{L}^{\infty,1}$. Using the conservation of the $L^2$ norm (the mass $M$), one obtains global well-posedness for data in $L^2$ (and afterwards in any Sobolev and weighted $L^2$ space).

Gaussians play the role of a ground state for the equation, since they minimize the Hamiltonian~$\mathcal{E}$ for fixed $M$. This variational characterization leads to orbital stability in $L^2$ (up to the symmetry group of the equation, which will be recalled in Section~\ref{sectionsymmetries}); it also holds in $L^{2,1} \cap H^1$ (where $L^{2,1}$ is given by the norm $\|f\|_{L^{2,1}} = \| \langle x \rangle f \|_{L^2}$) by a different argument.

\subsection{Obtained results}

\subsubsection{The basis of special Hermite functions}

Recall that there exists a Hilbertian basis of $L^2(\R^2)$ known as  the special Hermite functions $\{ \phi_{n,m} \}$, where $n \in \mathbb{N}, m \in \{-n,2-n,\dots,n-2,n\}$ which  diagonalizes jointly the harmonic oscillator $H = - \Delta + |x|^2$ and the angular momentum operator~${L = i x \times \nabla}$:
$$
H \phi_{n,m} = 2(n+1) \phi_{n,m}, \qquad L \phi_{n,m} = m \phi_{n,m}
$$
(see Section~\ref{sectionSHF} for a more detailed presentation). We show in Section~\ref{sectionSHF} that this basis is very well-suited to decomposing the trilinear operator $\mathcal{T}$: this follows from the formula (proved in Proposition~\ref{propFormule})
$$
\mathcal{T}(\phi_{n_1,m_1},\phi_{n_2,m_2},\phi_{n_3,m_3}) =  \mathcal{E}(\phi_{n_1,m_1},\phi_{n_2,m_2},\phi_{n_3,m_3},\phi_{n_4,m_4}) \phi_{n_4,m_4},
$$
where $n_4 = n_1 + n_2 - n_3$ and $m_4 = m_1 + m_2 - m_3$. This formula has the immediate consequence that all the special Hermite functions are stationary waves. Furthermore, it also implies the dynamical invariance of all the subspaces of the form
$$
\operatorname{Span}_{L^2} \left\{ \phi_{n,m}, \;\mbox{such that}\; \alpha n + \beta m = \gamma \;\operatorname{mod} \; \delta \right\} \quad \mbox{or} \quad \operatorname{Span}_{L^2} \left\{ \phi_{n,m},\; \mbox{such that}\; \alpha n + \beta m = \gamma \right\}
$$
where $\alpha$, $\beta$, $\gamma$ and $\delta$ are natural numbers. 

\subsubsection{Invariant subspaces} We investigate further the dynamics on particularly relevant or natural examples of the invariant subspaces which were just described:
\begin{itemize}
\item Sections~\ref{sect41}--\ref{sect44} are dedicated to the analysis of the dynamics on the eigenspaces of the harmonic oscillator $H$: for some $n_0 \in \mathbb{N}$, $\operatorname{Span}_{L^2} \left\{ \phi_{n_0,m}, \mbox{where}\; m \in \{-n_0,2-n_0,\dots,n_0-2,n_0\} \right\}$.
\item Section~\ref{sectioneigenL} focuses on the eigenspaces of the rotation operator $L$: for some $m_0 \in \mathbb{Z}$, \\$\operatorname{Span}_{L^2} \left\{ \phi_{n,m_0}, \mbox{where}\; n \in \{ |m_0|, |m_0|+2, |m_0| + 4 ,\dots\} \right\}$.
\item Finally, in Section~\ref{sectionbargmannfock}, we study the equation on the Bargmann-Fock space which can be seen as $\operatorname{Span}_{L^2} \left\{ \phi_{n,n}, \mbox{where} \; n \in \mathbb{N} \right\}$.
\end{itemize}

\subsubsection{Stationary waves} The most classical notion of a stationary wave is given by a solution of the type $e^{-i \omega t} \phi$, where $\phi$ is a fixed function, and $\omega \in \mathbb{R}$. More elaborate waves are of the type $R_{-\alpha \omega t} e^{-i\omega t} \phi$, where $R_\theta$ is the rotation operator in space around $0$ of angle $\theta$, and $\alpha$, $\omega$ real numbers. For reasons that will become clear, we call the former $M$-stationary waves (or simply stationary waves), and the latter $M+\alpha P$-stationary waves.

In the invariant subspaces detailed above, we give examples of stationary solutions, and try to investigate their stability. Understanding the full picture - even finding all stationary solutions - seems a daunting task, but we obtain first results in this direction.

Finally, in Section~\ref{sectionstationary}, we prove general theorems on decay and regularity of stationary waves of  \eqref{CR}. More precisely, we show that any $M$-stationary wave in $L^2$ is analytic and exponentially decreasing, while, roughly speaking, $M + \alpha P$ stationary waves belong to the Schwartz class as soon as they are a little more localized, and smooth, than $L^2$.

\subsection{Notations}

We set
\begin{itemize}
\item[$\bullet$] For $x \in \mathbb{R}^2$, $x^\perp$ is the rotation of $x$ by $\frac{\pi}{2}$ around the origin.
\item[$\bullet$] For $x \in \mathbb{R}^2$, $\langle x \rangle = \sqrt{1 + |x|^2}$; similarly if $x \in \mathbb{R}$.
\item[$\bullet$] $\langle f\,,\, g \rangle_{L^2(\mathbb{R}^2)} \overset{def}{=} \int_{\mathbb{R}^2} f \overline{g}$.
\item[$\bullet$] $\displaystyle \mathcal{F} f (\xi) = \widehat{f}(\xi) \overset{def}{=} \frac{1}{2\pi} \int_{\mathbb{R}^2} f(x) e^{-ix\xi}\,dx$.
\item[$\bullet$] $\N$ is the set of all non-negative integers (including 0).
\item[$\bullet$] $H \overset{def}{=} -\Delta + |x|^2$ is the quantum harmonic oscillator on $\R^2$.
\item[$\bullet$] $L \overset{def}{=} i(x_2 \partial_{x_1} - x_1 \partial_{x_2})$ is the angular momentum operator.
\item[$\bullet$] $H^s$ is the Sobolev space given by the norm $\|f\|_{H^s}  \overset{def}{=} \| \langle \xi \rangle^s \widehat f \|_{L^2}$.
\item[$\bullet$] $L^{2,s}$ is the weighted $L^2$-space given by the norm $\|f\|_{L^{2,s}} \overset{def}{=}  \| \langle x \rangle^s f \|_{L^2}$.
\item[$\bullet$] $\H^s=H^{s} \cap L^{2,s}$ is the weighted Sobolev space given by the norm $\|f\|_{\H^s}  \overset{def}{=}   \| \langle x \rangle^s f \|_{L^2}+\| \langle \xi \rangle^s \widehat f \|_{L^2}$. It is classical (see \eqref{eq}) that $\H^s=\big\{u\in L^{2},\;\;{\it s.t.}\;\; H^{s/2}u\in L^{2}\big\}$.
\item[$\bullet$] $R_\theta$ is the counter-clockwise rotation of angle $\theta$ around the origin.
\item[$\bullet$] $S_\lambda u = \lambda u (\lambda \cdot)$ is the $L^2$ scaling.
\item[$\bullet$] $\Pi_n$ is the orthogonal projection on the eigenspace $E_n=\big\{ u\in L^2(\R^2),\; Hu=2(n+1)u\,\big\} $.
\end{itemize}

In this paper $c,C>0$ denote universal constants the value of which may change from line to line. For two quantities $A$ and $B$, we denote $A \lesssim B$ if $A \leq CB$, and $A \approx B$ if $A \lesssim B$ and $A \gtrsim B$.

\section{\texorpdfstring{Properties and symmetries of  $\mathcal{T}$ and $\mathcal{E}$}{Properties and symmetries of  T and E}}

\subsection{\texorpdfstring{Various formulations of $\mathcal{T}$ and $\mathcal{E}$}{Various formulations of T and E}} 
\label{sectionsymmetries}

As observed in~\cite{FGH} we have

\begin{lemm}\label{FourierSym}The quantities $\mathcal{T}$ and $\mathcal{E}$ are invariant by Fourier transform:
\begin{equation}\label{fourier}
\F \big(  \T(f_1,f_2,f_3)\big)= \T(\wh{f_1},\wh{f_2},\wh{f_3}) \quad \mbox{and} \quad \mathcal{E}(f_1,f_2,f_3,f_4) = \mathcal{E}(\wh{f_1},\wh{f_2},\wh{f_3},\wh{f_4}). 
\end{equation}
\end{lemm}
\begin{proof}
Using Fourier inversion and the identity $\frac{1}{4\pi^2} \int e^{ix \xi} \,dx = \delta_{\xi = 0}$ gives
\begin{equation*}
\begin{split}
\mathcal{F} \mathcal{T} (f,g,h)(\xi) & = \frac{1}{2\pi} \int_{\mathbb{R}^2} \int_\mathbb{R} \int_{\mathbb{R}^2} e^{-iz\xi} f(x+z) g(\lambda x^\perp + z) 
\overline{h(x+\lambda x^\perp +z)}\,dx\,d\lambda \,dz \\
& = \frac{1}{16 \pi^4} \int \int \int \int  \int \int e^{iz(-\xi+\alpha+\beta-\gamma)} e^{ix(\alpha-\lambda \beta^\perp + \lambda
\gamma^\perp - \gamma)} \widehat{f}(\alpha) \widehat{g} (\beta) \overline{\widehat{h}(\gamma)} \,d\alpha \,d\beta \,
d\gamma \,dx\,d\lambda \,dz \\
& = \int \int \widehat{f} \left( \xi + \frac{1}{\lambda} \xi^\perp - \frac{1}{\lambda} \beta^\perp \right) \widehat{g}(\beta) 
\overline{\widehat{h} \left( \frac{1}{\lambda} \xi^\perp + \beta - \frac{1}{\lambda} \beta^\perp \right)}\,d\beta \,
\frac{d \lambda}{\lambda^2}.
\end{split}
\end{equation*}
Changing variables to $\eta = \frac{1}{\lambda} ( \xi - \beta)^\perp$ gives
$$
\mathcal{F} \mathcal{T} (f,g,h)(\xi) = \int_{\mathbb{R}} \int_{\mathbb{R}^2} \widehat{f} (\xi + \eta) \widehat{g} (\xi+\lambda \eta^\perp)
\overline{\widehat{h} ( \eta + \lambda \eta^\perp + \xi )} \,d\eta\,d\lambda,
$$
which is the desired result.
\end{proof}

 The next result shows that $\mathcal{E}$ can be related to the $L_t^4L_x^4$ Strichartz norm associated to the linear flows $\e^{it \Delta}$ and $\e^{-it H}$.
\begin{lemm}\ph\label{lem11}
The following formulations for $\mathcal{E}$ hold
\begin{eqnarray}
\mathcal E(f_1,f_2,f_3,f_4)&=& 2\pi  \int_{\R}\int_{\R^2}  (\e^{it \Delta}f_1)   (\e^{it \Delta}f_2) (\ov{ \e^{it \Delta}f_3}) (\ov{\e^{it \Delta}f_4}) dx\, dt \label{eq1}\\
&=& 2\pi  \int_{-\frac{\pi}4}^{\frac{\pi}4}\int_{\R^2}  (\e^{-it H}f_1)   (\e^{-it H}f_2) (\ov{ \e^{-it H}f_3}) (\ov{\e^{-it H}f_4}) dx\, dt\label{eq2}.
\end{eqnarray}
Therefore we have
\begin{eqnarray}
\T(f_1,f_2,f_3)&=&2\pi \int_{\R}   \e^{-it \Delta}  \Big[ (\e^{it \Delta}f_1) ({ \e^{it \Delta}f_2}) (\ov{\e^{it \Delta}f_3})\Big] dt\nonumber\\
&=&2\pi  \int_{-\frac{\pi}4}^{\frac{\pi}4}   \e^{it H}  \Big[ (\e^{-it H}f_1) ({ \e^{-it H}f_2}) (\ov{\e^{-it H}f_3})\Big] dt \label{eq3}.
\end{eqnarray}
\end{lemm}

\begin{proof}
Using Fourier inversion and the identity $\frac{1}{4\pi^2} \int_{\mathbb{R}^2} e^{ix \xi} \,dx = \delta_{\xi = 0}$ gives
\begin{align*}
&A  \overset{def}{=} \int_{\mathbb{R}} \int_{\mathbb{R}^2} (\e^{it \Delta}f_1)   (\e^{it \Delta}f_2) (\ov{ \e^{it \Delta}f_3}) (\ov{\e^{it \Delta}f_4}) dx\, dt \\
& \quad = \frac{1}{(2\pi)^4} \int \int \int \int \int \int e^{-it(|\alpha|^2 + |\beta|^2 - |\gamma|^2 - |\delta|^2)} e^{ix(\alpha + \beta - \gamma - \delta)} \widehat{f_1}(\alpha) \widehat{f_2}(\beta) \overline{\widehat{f_3}(\gamma)}\,\overline{\widehat{f_4}(\delta)} \,d\alpha\,d\beta\,d\gamma\,d\delta\,dx\,dt \\
& \quad = \frac{1}{(2\pi)^2} \int \int \int \int e^{-it(|\alpha|^2 + |\beta|^2 - |\alpha+\beta-\delta|^2 - |\delta|^2)} \widehat{f_1}(\alpha) \widehat{f_2}(\beta) \overline{\widehat{f_3}(\alpha + \beta - \delta)}\,\overline{\widehat{f_4}(\delta)} \,d\alpha\,d\beta\,d\delta \,dt.
\end{align*}
Changing variables to $\delta =z$, $\alpha = z+x$, $\beta = z + \lambda x^\perp + \mu x$ and resorting to the identity $\frac{1}{2\pi} \int_{\mathbb{R}} e^{iy \xi} \,dy = \delta_{\xi = 0}$ yields
\begin{align*}
A & = \frac{1}{(2\pi)^2} \int \int \int \int \int e^{2it\mu |x|^2} \widehat{f_1}(z+x) \widehat{f_2}(z + \lambda x^\perp + \mu x) \overline{\widehat{f_3}( z + x + \lambda x^\perp + \mu x)}\, \overline{\widehat{f_4}(z)} |x|^2 \,dx\,dz\,d\lambda\,d\mu\,dt \\
& = \frac{1}{2\pi} \int \int \int \widehat{f_1}(x+z) \widehat{f_2}(z + \lambda x^\perp) \overline{\widehat{f_3}(z + x + \lambda x^\perp)} \, \overline{\widehat{f_4}(z)}\,dx\,dz\,d\lambda\\
& = \frac{1}{2\pi} \mathcal{E}(\widehat{f_1}, \widehat{f_2},\widehat{f_3} , \widehat{f_4}) = \frac{1}{2\pi} \mathcal{E}(f_1,f_2,f_3,f_4),
\end{align*}
which gives~(\ref{eq1}). Let us now prove \eqref{eq2}. Let $f\in L^2(\R^2)$, and denote $v(t,\cdot)=\e^{-itH}f$ and ${u(t,\cdot)=\e^{it\Delta}f}$. Then the lens transform gives (see for instance \cite{Tao}) 
\begin{equation}
\label{lenstransf}
u(t,x) =    \frac{1}{\sqrt{1+4t^2}}    v \Big( \frac{\arctan(2t)}{2}  , \frac{x}{\sqrt{1+4t^2} } \Big)     \e^{ \frac{i|x|^2t}{1+4t^2}  }.
\end{equation}
We first make the change of variables $\dis y= \frac{x}{\sqrt{1+4t^2} }$, then $\dis \tau= \frac{\arctan(2t)}{2}$. This gives   
 \begin{eqnarray*}
 \mathcal E(f_1,f_2,f_3,f_4)&=& 2\pi  \int_{\R}  \frac{1}{1+4t^2}\int_{\R^2}[v_1v_2\ov{v_3}\ov{v_4}]\left( \frac{\arctan(2t)}{2},y \right)   dy\, dt  \\
 &=&  2\pi  \int_{-\frac{\pi}4}^{\frac{\pi}4}\int_{\R^2}  (\e^{-i\tau H}f_1)   (\e^{-i\tau H}f_2) (\ov{ \e^{-i\tau H}f_3}) (\ov{\e^{-i\tau H}f_4}) dy\,d\tau. 
 \end{eqnarray*}
The relations for $\T$ are obtained using that $\<  \T(f_1,f_2,f_3),f_4  \>_{L^2(\R^2)}= \mathcal{E}(f_1,f_2,f_3,f_4)$.
\end{proof}
 We are now able to prove the following result
 \begin{lemm}\ph \label{lemproj}
The following formulations for $\mathcal{E}$ hold
\begin{equation*}
\mathcal E(f_1,f_2,f_3,f_4)= \pi^2 \sum_{n_1+n_2=n_3+n_4}   \int_{\R^2}  (\Pi_{n_1}f_1)   (\Pi_{n_2}f_2) (\ov{ \Pi_{n_3}f_3}) (\ov{ \Pi_{n_4}f_4}) dx.
\end{equation*}
Therefore we have
\begin{equation*}
\T(f_1,f_2,f_3)=\pi^2 \sum_{n_1+n_2=n_3+n_4}  \Pi_{n_4}\Big( (\Pi_{n_1}f_1)   (\Pi_{n_2}f_2) (\ov{ \Pi_{n_3}f_3}) \Big).
\end{equation*}
\end{lemm}

\begin{proof}
We compute $\mathcal{E}$  in \eqref{eq2} for the eigenfunctions of $H$. Therefore we assume that $\Pi_{n_j}f_j=f_j$ and then 
\begin{equation*}
\mathcal E(f_1,f_2,f_3,f_4)= 2\pi \int_{-\frac{\pi}4}^{\frac{\pi}4}\e^{-2i(n_1+n_2-n_3-n_4)t}dt  \int_{\R^2} f_1f_2\ov{ f_3f_4}dx.
\end{equation*}
But now we use that $\Pi_{n_j}f(-x)=(-1)^{n_j}\Pi_{n_j}f(x)$, thus $\int_{\R^2} f_1f_2\ov{ f_3f_4}dx=0$ unless $n_1+n_2-n_3-n_4$ is even, which in turn implies $\int_{-\frac{\pi}4}^{\frac{\pi}4}\e^{-2i(n_1+n_2-n_3-n_4)t}dt=\frac{\pi}{2} \delta(n_1+n_2-n_3-n_4)$.
\end{proof}

\subsection{\texorpdfstring{Symmetries of $\mathcal{T}$ and conservation laws for \eqref{CR}}{Symmetries of T and conservation laws for (CR)}} 
\begin{lemm}\ph\label{lemconj}
The commutation relation
\begin{equation}
\label{commute}
Q\big(\T(f_1,f_2,f_3)\big)= \T(Qf_1,f_2,f_3) +\T(f_1,Qf_2,f_3) -\T(f_1,f_2,Qf_3)
\end{equation}
holds for the (self-adjoint) operators
$$
Q = 1 \;,\; x \;,\; |x|^2 \;,\; i\nabla \;,\; \Delta \;,\; H \;,\; L = ix \times \nabla \;,\; i(x\cdot \nabla + 1)
$$
for all $f_1,f_2,f_3 \in \mathcal{D}(Q)$, where $\mathcal{D}(Q)$ denotes the domain of $Q$.
\end{lemm}

\begin{proof}
First observe that it suffices to prove the commutation relation for $f_1,f_2, f_3$ sufficiently smooth, and then argue by density and $L^2$ boundedness of $\mathcal{T}$.
For $Q= 1$, $x$, $|x|^2$, this follows easily from the definition~\eqref{lim} of $\mathcal{T}$, in particular the fact that the arguments of $ f_1, f_2, f_3$ satisfy
\begin{align*}
& z + (x + \lambda x^\perp + z) = (x+z) + (\lambda x^\perp + z) \\
& |z|^2 + |x + \lambda x^\perp + z|^2 = |x+z| + |\lambda x^\perp + z|^2.
\end{align*}
Using~\eqref{fourier} and arguing similarly gives \eqref{commute} for $P= i\nabla$ and $\Delta$. Combining $|x|^2$ and $\Delta$ gives~(\ref{commute}) for $Q=H$. Finally, to obtain this commutation relation for $Q=ix \times \nabla$ and $i(x\cdot \nabla + 1)$, define $R_\lambda$ to be the rotation of angle $\lambda$ around the origin, and $S_\lambda$ the scaling transformation $S_\lambda u = \lambda u (\lambda \cdot)$, observe that
\begin{align*}
& R_\lambda \mathcal{T}(u,u,u)= \mathcal{T} (R_\lambda u, R_\lambda u,R_\lambda u) \\
& S_\lambda \mathcal{T}(u,u,u)= \mathcal{T} (S_\lambda u, S_\lambda u,S_\lambda u)
\end{align*}
and differentiate in $\lambda$.
\end{proof}

\begin{coro}
\label{loriot}
If $Q$ is as in Lemma~\ref{lemconj}, and $f_1,f_2,f_3 \in L^2(\mathbb{R}^2)$ are such that $Q f_j = \lambda_{j} f_j$, with $\lambda_j \in \mathbb{C}$ and $j=1,2,3$, then
$$
Q \big(\T(f_1,f_2,f_3)\big) = (\lambda_{1}+\lambda_{2}-\lambda_{3})\T(f_1,f_2,f_3)
$$
and for $s \in \mathbb{R}$,
$$
e^{isQ}  \T(f_1,f_2,f_3) = \T(e^{isQ}f_1,e^{isQ}f_2,e^{isQ}f_3) \quad \mbox{and} \quad \mathcal{E}(e^{isQ} f) = \mathcal{E}(f) .
$$
\end{coro}

In particular, if $f_1, f_2$, and $f_3$ are eigenfunctions of $Q$ as in Lemma~\ref{lemconj}, then so is $\T(f_1, f_2, f_3)$. This corollary hints towards examining $\mathcal{T}$ in a basis which simultaneously diagonalizes both $H$ and the angular momentum operator $L = i( x\times \nabla)=  i(x_2 \partial_{x_1} - x_1 \partial_{x_2})$; this will be done in the next section.

\begin{lemm}\ph If $Q$ is an  operator so that for all $f\in \mathcal{S}(\R^2)$
\begin{equation*}
Q\big(\T(f,f,f)\big)= 2\T(Qf,f,f) -\T(f,f,Q^{\star}f),
\end{equation*}
then 
\begin{equation*}
\int_{\R^2} (Qu) \ov{u}
\end{equation*}
is a conservation law for \eqref{CR}. Applying this to
$$
Q = 1 \;,\; x \;,\; |x|^2 \;,\; i\nabla \;,\; \Delta \;,\; H \;,\; L \;,\; i(x\cdot \nabla + 1)
$$
gives the conserved quantities
$$
M= \int |u|^2 \;,\; \int x |u|^2 \;,\; \int|x|^2 |u|^2 \;,\; \int i\nabla u \,\overline{u} \;,\; \int |\nabla u|^2 \;,\; \int H u \,\overline{u} \;,\; P=\int L u \,\overline{u} \;,\; \int i(x\cdot \nabla + 1) u \,\overline{u}
$$
(which are real-valued since we chose $Q$ self-adjoint).
\end{lemm}

\begin{proof}
We compute
\begin{equation}\label{deri}
\frac{d}{dt}\int_{\R^2} (iQu) \ov{u}=\<Q\T(u),u\>-\<Qu,\T(u)\>.
\end{equation}
 By assumption and the symmetries of $\mathcal{E}$
 \begin{eqnarray*}
\<Q\T(u),u\>&=&2 \<\T(Qu,u,u),u\>-\<\T(u,u,Q^{\star}u),u\>\\
 &= &2\<Qu,\T(u)\>- \<\T(u),Q^{\star}u\>,
 \end{eqnarray*}
 which implies that $ \<\T(u),Q^{\star}u\>=\<Qu,\T(u)\>,$ and yields the result by \eqref{deri}.
\end{proof}

The relation which has just been established between operators commuting with $\mathcal{T}$, symmetries of $\mathcal{T}$ and $\mathcal{E}$, and conserved quantities of  \eqref{CR} is of course an instance of the Noether theorem. We recapitulate below the obtained results (with $\lambda\in \R$).
\medskip
\begin{center}
\begin{tabular}{| c | c | c | c |}
\hline 
operator $Q$ & conserved quantity & corresponding \\
commuting with $\mathcal{T}$ & $\int Q u \overline{u}$ & symmetry $u\mapsto e^{i\lambda Q} u$ \\[2pt]
\hline\vspace{1pt}
$1$ & $\int |u|^2$ & $u \mapsto e^{i\lambda} u$ \\[2pt]
\hline \vspace{1pt}
$x_1$ & $\int x_1 |u|^2$ &  $u \mapsto e^{i\lambda x_1} u$ \\[3pt]
\hline\vspace{1pt}
$x_2$ & $\int x_2 |u|^2$ &  $u \mapsto e^{i\lambda x_2} u$ \\[2pt]
\hline\vspace{1pt}
$|x|^2$ & $\int |x|^2 |u|^2$  & $u \mapsto e^{i\lambda \vert x\vert ^2} u$ \\[2pt]
\hline\vspace{1pt}
$i\partial_{x_1}$ & $\int i\partial_{x_1} u \overline{u}$  & $u \mapsto u(\cdot +\lambda e_1)$ \\[2pt]
\hline\vspace{1pt}
$i\partial_{x_2}$ & $\int  i\partial_{x_2} u \overline{u}$ & $u \mapsto u(\cdot + \lambda e_2)$ \\[2pt]
\hline\vspace{1pt}
$\Delta$ & $\int |\nabla u|^2$ & $u \mapsto e^{i \lambda \Delta} u$ \\[2pt]
\hline\vspace{1pt}
$H$ & $\int Hu \overline{u}$ & $u \mapsto e^{i \lambda H} u$\\[2pt]
\hline\vspace{1pt}
$L = i(x \times \nabla)$ & $\int L u \overline{u}$ & $u \mapsto u(R_\lambda x)$ \\[2pt]
\hline\vspace{1pt}
$i(x \cdot \nabla +1)$ & $\int i(x \cdot \nabla+1) u \overline{u}$ & 
$u \mapsto \lambda u( \lambda x)$  \\[2pt]
\hline
\end{tabular}
\end{center}

\section{Approximation of NLS with harmonic trapping by (CR)}

A first consequence of the findings of the previous section is the following theorem, which states that \eqref{CR} approximates the dynamics of \eqref{nls} in the small data regime. For $s\geq 0$, we define the Sobolev space based on the harmonic oscillator $\H^s=\big\{u\in L^{2},\;\;{\it s.t.}\;\; H^{s/2}u\in L^{2}\big\}$, endowed with the natural norm     $\Vert u\Vert_{\H^{s}}=\|H^{s/2}u\|_{L^{2}}$. By~\cite[Lemma~2.4]{YajimaZhang2}, we have the following equivalence of norms
             \begin{equation}\label{eq}
                 \Vert u\Vert_{\H^{s}}   \equiv        \Vert u\Vert_{H^{s}}+ \Vert   \<x\>^{s}u\Vert_{L^{2}}. 
  \end{equation}

\begin{theo} \label{approximation theorem}
Let $s>1$ and suppose that $u(t)$ is a solution of \eqref{nls} and $f(t)$ a solution of \eqref{CR} with the same initial data $u_0$. Assume that the following bound holds over an interval of time $[0,T]$
$$
\|f(t)\|_{\H^s} \leq B\qquad \text{for all } \;\;0\leq t\leq T.
$$
Then there exists a constant $C>1$ such that 
$$
\|u(t)-e^{itH} f(\frac{t}{\pi^2})\|_{\H^s(\R^2)}\leq C(B^3+B^5t) e^{CB^2t}\qquad \text{for all }\;\; 0\leq t\leq T.
$$
In particular, if $B$ is sufficiently small and $0\leq t \ll B^{-2}\log B^{-1}$, then 
$$
\|u(t)-e^{itH} f(\frac{t}{\pi^2})\|_{\H^s(\R^2)} \leq B^{5/2}.
$$
 \end{theo} 
We remark that $B$ can be made small by taking the initial condition sufficiently small in $\H^s$.

\begin{proof}
We start with some notation: we write $\omega= n_1+n_2-n_3-n_4$ and 
\begin{align*}
\T^\prime(f_1,f_2,f_3)=&\sum_{\substack {n_1,n_2,n_3,n_4\geq 0\\\omega=0}}  \Pi_{n_4}\Big( (\Pi_{n_1}f_1)   (\Pi_{n_2}f_2) (\ov{ \Pi_{n_3}f_3}) \Big),\\
\mathcal N_t(f_1, f_2, f_3)=&\sum_{\substack {n_1,n_2,n_3,n_4\geq 0}}  e^{it\omega}\Pi_{n_4}\Big( (\Pi_{n_1}f_1)   (\Pi_{n_2}f_2) (\ov{ \Pi_{n_3}f_3}) \Big)=e^{-itH} (e^{itH} f_1\,e^{itH} f_2\, \overline{e^{itH} f_3}),\\
\mathcal P_t(f_1,f_2,f_3)=&\sum_{\substack {n_1,n_2,n_3,n_4\geq 0\\\omega\neq 0}}  e^{it\omega} \Pi_{n_4}\Big( (\Pi_{n_1}f_1)   (\Pi_{n_2}f_2) (\ov{ \Pi_{n_3}f_3}) \Big) = \mathcal N_t(f_1,f_2,f_3) - \T^\prime(f_1,f_2,f_3).
\end{align*}
Note that $\T'$ only differs from $\T$ by a constant multiplicative factor (see Lemma~\ref{lemproj}).
Recall that for $s>1$,  we have  $H^{s}(\R^{2}) \subset L^{\infty}(\R^{2})$, and that $H^{s}(\R^{2})$ is an algebra. Then  by \eqref{eq}, it is easy to check that $\H^{s}(\R^{2})$ is also an algebra when $s>1$. Therefore, using the boundedness of the operator $e^{itH}$ on $\H^s$ (uniformly in $t$), we obtain the boundedness of the trilinear operators $\T', \mathcal N$, and $\mathcal P$ from $\H^s\times \H^s \times \H^s$ to $\H^s$ for any $s>1$. 
\medskip

Let $g(t)= e^{-itH} u(t)$ and $\widetilde f(t)=f(\frac{t}{\pi^2})$. The equation satisfied by $g(t)$ and $\tilde f$ are the following:
$$
i\partial_t g(t)= \mathcal N_t(g(t),g(t),g(t)); \qquad i\partial_t \widetilde f=\T'(\widetilde f,\widetilde f,\widetilde f);\qquad g(0)=f(0)=u_0.
$$

Now we write the equation for $g(t)$ as follows:
\begin{align}
i\partial_t g=& \T^\prime(g,g,g)+ \sum_{\substack {n_1,n_2,n_3,n_4\geq 0\\\omega\neq 0}}  e^{it\omega}\Pi_{n_4}\Big( (\Pi_{n_1}g)   (\Pi_{n_2}g) (\ov{ \Pi_{n_3}g}) \Big)\nonumber \\
  \label{hl1} =&\T^\prime(g,g,g)+ \partial_t \sum_{\substack {n_1,n_2,n_3,n_4\geq 0\\\omega\neq 0}}  \frac{e^{it\omega}-1}{i\omega}\Pi_{n_4}\Big( (\Pi_{n_1}g)   (\Pi_{n_2}g) (\ov{ \Pi_{n_3}g}) \Big)\\
 \label{hl2} &\qquad -\sum_{\substack {n_1,n_2,n_3,n_4\geq 0\\\omega\neq 0}}  \frac{e^{it\omega}-1}{i\omega}\Pi_{n_4}\partial_t \Big( (\Pi_{n_1}g)   (\Pi_{n_2}g) (\ov{ \Pi_{n_3}g}) \Big)\\ 
 =& \T^\prime (g,g,g) +\partial_t A +D \nonumber,
\end{align}
where $A$ is the sum in \eqref{hl1} and $D$ is given in \eqref{hl2}. Now denote by $[t]$ the largest integer smaller than~$t$ and notice that 
\begin{align*}
&\sum_{\substack {n_1,n_2,n_3,n_4\geq 0\\\omega\neq 0}}  \frac{e^{it\omega}-1}{i\omega}\Pi_{n_4}\Big( (\Pi_{n_1}g)   (\Pi_{n_2}g) (\ov{ \Pi_{n_3}g}) \Big)=\sum_{\substack {n_1,n_2,n_3,n_4\geq 0\\\omega\neq 0}}  \int_{2\pi[\frac{t}{2\pi}]}^{t}e^{i\tau \omega}\Pi_{n_4}\Big( (\Pi_{n_1}g)   (\Pi_{n_2}g) (\ov{ \Pi_{n_3}g}) \Big)\, d\tau\\
&=\int_{2\pi[\frac{t}{2\pi}]}^{t} \mathcal P_\tau( g(t), g(t),g(t)) \, d\tau.
\end{align*}
This allows to estimate 
\begin{equation*} 
\begin{split}
\|A\|_{\H^s}\lesssim& \|g\|^{3}_{\H^s}, \\
\|D\|_{\H^s}\lesssim& \|\partial_t g\|_{\H^s}\|g\|_{\H^s}^2\lesssim \|g\|_{\H^s}^5.
\end{split}
\end{equation*}
Now let $e(t)= g(t)-\widetilde f(t)$, and assume as a bootstrap hypothesis that $\|e(t)\|_{\H^s}\leq B$. Then the equation satisfied by $e(t)$ can be written as
$$
\partial_t e= \T'(e, g,g )+\T'(\widetilde f, e, g) +\T'(\widetilde f,\widetilde f,e)+\partial_t A +D \nonumber.
$$ 
This gives that 
$$
\|e(t)\|_{\H^s} \leq C_1 B^2\int_0^t \| e(s)\|_{\H^s} ds +C_1(B^3+ t B^5).
$$
The result now follows by Gronwall's inequality.
\end{proof}


\section{Regularity and decay of stationary waves}

\label{sectionstationary}
Recall the following conservation laws
\begin{equation*}
M(u)=\int_{\R^2} |u|^2,\qquad P(u)= \int_{\R^2}Lu\, \ov{u}. 
\end{equation*}
As we mentioned in the introduction,   $M+\alpha P$-stationary waves read $R_{-\alpha \omega t} e^{-i\omega t} \phi$, where $\alpha$ and $\omega$ are real numbers, and $\phi$ is a fixed function. In particular, $M$-stationary waves are simply of the type $e^{-i\omega t} \phi$.
Notice that, in degenerate cases (if $\phi$ is an eigenfunction of $L$), a given solution can be an $M+\alpha P$-stationary wave for all\footnote{It would be natural to define $P$-stationary waves as waves of the type $R_{-\alpha \omega t} \phi$. However, we could not find an example of such a wave which is not at the same time an $M$-stationary wave. We conjecture any $P$-stationary wave is also an $M$-stationary wave.} $\alpha$.
Variationally, $M+\alpha P$-stationary waves can be characterized as critical points of $\mathcal{E}$ under the constraint that $M+\alpha P$ takes a fixed value. The Euler-Lagrange equation reads
\begin{equation}
\label{EL}
\omega \phi + \alpha \omega L \phi = \mathcal{T}(\phi,\phi,\phi).
\end{equation}

In the following subsections, we prove polynomial or exponential decay, in the physical space or Fourier variable, for large classes of such stationary waves, requiring that they belong to $L^2$, or slightly more.

On the one hand, the condition of belonging to $L^2$ is sharp in order to obtain decay or regularity. This can be seen through the example of $\frac{1}{|x|}$, which is a stationary wave in the weak-$L^2$ space $L^{2,\infty}$, but not in any Sobolev or weighted $L^2$ space of positive index. On the other hand, all the examples we know of $M+\alpha P$ stationary waves are products of Gaussians and polynomials, suggesting a possible improvement of our results.

\subsection{Polynomial decay}
\begin{theo} \label{theo31} Assume that $\alpha \in \mathbb{R}$, with $\alpha = 0$ or $\alpha^{-1} \notin \mathbb{Z}$, and let $\phi$ be an $M + \alpha P$ stationary wave. Then
\begin{enumerate}[(i)]
\item If $\phi$ belongs to $L^{2,\epsilon}$, for some $\epsilon>0$, it belongs to $L^{2,s}$ for any $s>0$. 
\item If $\phi$ belongs to $H^{\epsilon}$, for some $\epsilon>0$, it belongs to $H^s$ for any $s>0$.
\item If $\phi$ belongs to $L^{2,\epsilon} \cap H^{\epsilon}$ for some $\epsilon>0$, it belongs to the Schwartz class.
\end{enumerate}
\end{theo}

\begin{proof} By invariance of the Euler-Lagrange equation~\eqref{EL} under the Fourier transform, it suffices to prove $(i)$. Using the fact that $(\operatorname{Id} + \alpha L)^{-1}$ is bounded on $L^{2,s}$ - since $L$ commutes with radial weights - it will follow from the repeated application of the following proposition to~\eqref{EL}. \end{proof}

\begin{prop} For $\sigma>0$, and $\delta<\frac{\sigma}{\sigma +1}$, the operator $\mathcal{T}$ is bounded from $(L^{2,\sigma})^3$ to $L^{2,\sigma+\delta}$.
\end{prop}

\begin{rema} This proposition also implies a very strong smoothing effect for the equation under study.
\end{rema}

\begin{proof} 
By duality, it suffices to prove that, for $f$ and $g$ in $L^2$, with norm 1, which are fixed from now on, and which we assume to be non negative,
$$
\langle \mathcal{T}( \langle x \rangle^{-\sigma} f\,,\, \langle x \rangle^{-\sigma} f\,,\, \langle x \rangle^{-\sigma} f )\,,\,\langle x \rangle^{\sigma+\delta} g \rangle \lesssim 1,
$$
or in other words
$$
\int_{-1}^1 \int_{\mathbb{R}^2} \int_{\mathbb{R}^2} K(x,z,\lambda) f(z+x) f(z + \lambda x^\perp) f(z+x+\lambda x^\perp) g(z) \,dx\,dz\,d\lambda \lesssim 1,
$$ 
where
$$
K(x,z,\lambda) = \frac{\langle z \rangle^{\sigma +\delta}}{\langle z+x \rangle^\sigma \langle z+\lambda x^\perp \rangle^\sigma \langle z+x+\lambda x^\perp \rangle^\sigma}.
$$

\bigskip
\noindent
\underline{Step 1: bounds on $K$.} Localizing dyadically $z$ and $z+x+\lambda x^\perp$, let us assume from now on that $\langle z \rangle \approx 2^j$ and $ \langle z+x+\lambda x^\perp \rangle \approx 2^k$, where $j$ and $k$ are integers.

First, it is clear that $K \lesssim 1$ unless $|z|$ is large. The identity
$$
|z|^2 + |z+x+\lambda x^\perp|^2 = |z+x|^2 + |z+\lambda x^\perp|^2
$$
implies that $\min(|z+x|, |z+\lambda x^\perp|) \geq |z|$. If in addition $|z+x+\lambda x^\perp| \gtrsim |z|$, the assumption $\delta<\sigma$ entails that $K\lesssim 1$.

Thus, $K\lesssim 1$ unless $2^j$ is large and $2^k < \epsilon_0 2^j$, for $\epsilon_0>0$ chosen sufficiently small. Let us assume for now that these two conditions hold. 

There holds then $|z+\lambda x^\perp| \approx |z|$. Indeed, $|z+\lambda x^\perp|\ll |z|$ would imply that $|\lambda x| \approx |z|$, and then $|z+x+\lambda x^\perp| \geq |x| - |z+\lambda x^\perp| \gtrsim |z|$ (since $|\lambda|<1$), which contradicts $2^k < \epsilon_0 2^j$.

Next, if $|\lambda x^\perp| > 2 \epsilon_0 2^j$, then $|z+x| \geq |\lambda x^\perp| - |z+x+\lambda x^\perp| > \epsilon_0 2^j$. Combined with $|z+\lambda x^\perp| \approx |z|$, and, once again $\delta<\sigma$, this implies $K \lesssim 1$.

Therefore, $K \lesssim 1$ unless $2^j$ large, $2^k < \epsilon_0 2^j$, and $|\lambda| |x| < 2\eps_0 2^j$. Under these three conditions,
$$
|z| \approx 2^j, \quad |x|\approx 2^j,\quad |z+\lambda x^\perp| \approx 2^j, \quad \langle z+x+\lambda x^\perp \rangle \approx 2^k < \epsilon_0 2^j.
$$
One can then estimate
$$
|z+x| \geq |\lambda| |x| - |z+x+\lambda x^\perp| \gtrsim |\lambda| 2^j \quad \mbox{provided $|\lambda| 2^j > C_0 2^k$, for a constant $C_0$}.
$$
When the previous estimate is valid, one finds by a straightforward computation
$$
|K| \lesssim \frac{2^{j(\sigma + \delta)}}{(|\lambda| 2^j)^\sigma 2^{j \sigma} 2^{k \sigma}} \lesssim 2^{j(\delta - \sigma)} |\lambda|^{-\sigma} 2^{-k\sigma},
$$
and this last bound is $O(1)$ if $|\lambda| > 2^{-(k+j)} 2^{j\frac{\delta}{\sigma}}$.

Summarizing: $K \lesssim 1$ unless $2^j$ large, $2^k < \epsilon_0 2^j$, and $|\lambda| < \max(C_0 2^{k-j},  2^{-(k+j)} 2^{j\frac{\delta}{\sigma}})$. If on the other hand these three conditions are satisfied, we use the trivial bound
\begin{equation}
\label{boundK}
|K| \lesssim \frac{2^{j(\sigma + \delta)}}{2^{j \sigma} 2^{k \sigma}} \lesssim 2^{j\delta} 2^{-k\sigma}.
\end{equation}

\medskip
\noindent
\underline{Step 2: Decomposition of $\mathcal{T}$.} Define 
\begin{align*}
&\chi_{j,k}(x,z,\lambda) = \mathbbm{1}_{2^j<\langle z \rangle<2^{j+1}} \mathbbm{1}_{2^k< \langle z+x+\lambda x^\perp \rangle<2^{k+1}} \mathbbm{1}_{\lambda < \max(C_0 2^{k-j},  2^{-(k+j)} 2^{j\frac{\delta}{\sigma}})}, \\
& \chi(x,z,\lambda) = \sum_{2^k < \epsilon_0 2^j} \chi_{j,k}(x,z,\lambda),
\end{align*}
and decompose
\begin{equation*}
\begin{split}
& \iiint K(x,z,\lambda) f(z+x) f(z + \lambda x^\perp) f(z+x+\lambda x^\perp) g(z) \,dx\,dz\,d\lambda \\
& \qquad = \iiint  \dots \chi(x,z,\lambda)  \,dx\,dz\,d\lambda +  \iiint  \dots (1-\chi(x,z,\lambda))  \,dx\,dz\,d\lambda \overset{def}{=} I + II.
\end{split}
\end{equation*}
Using that $|K| \lesssim 1$ on the support of the integrand of $II$, we obtain immediately, by boundedness of\;$\mathcal{T}$ from $(L^2)^3$ to $L^2$,
$$
II \lesssim \big| \langle \mathcal{T}(f,f,f),g \rangle\big| \lesssim 1.
$$
As for $I$, the bound~\eqref{boundK} gives
$$
I \lesssim \sum_{2^k < \epsilon_0 2^j} 2^{j\delta} 2^{-k\sigma}\big| \langle \mathcal{T}_{j,k}(f,f,f),g \rangle\big|,
$$
where
$$
\mathcal{T}_{j,k}(f,g,h)(z) = \int_{-1}^1 \int_{\mathbb{R}^2} f(z+x) g(z + \lambda x^\perp) h(z+x+\lambda x^\perp) \chi_{j,k}(x,z,\lambda) \,dx \,d\lambda.
$$
By  \cite[Proposition 7.7]{FGH},  the operator $\mathcal{T}_{j,k}  : (L^2)^3 \to L^2$ has an operator  norm such that $\Vert  \mathcal{T}_{j,k}  \Vert_{ (L^2)^3 \to L^2} \leq  \max(C_0 2^{k-j},  2^{-(k+j)} 2^{j\frac{\delta}{\sigma}})$. This implies that
$$
I \lesssim  \sum_{2^k < \epsilon_0 2^j} 2^{j\delta} 2^{-k\sigma} \max(C_0 2^{k-j},  2^{-(k+j)} 2^{j\frac{\delta}{\sigma}}) \lesssim 1
$$
(since $\delta < \frac{\sigma}{\sigma +1}$), which concludes the proof.
\end{proof}

\subsection{Exponential decay}

The purpose of this section is to prove analyticity and exponential decay of $M-$stationary solutions of~\eqref{CR}. Indeed, we will show the following:

\begin{theo}\label{analyticity}
 Any $L^2$ solution of the equation
\begin{equation}
\label{EL0}
\omega \phi = \mathcal{T}(\phi,\phi,\phi).
\end{equation}
satisfies

$$
e^{\mu |\cdot|} \phi, \;e^{\mu |\cdot|} \widehat \phi \in L^\infty
$$
for some $\mu>0$.
\end{theo}

As a corollary, one gets that $\phi$ is analytic in a complex strip $\{(z_1, z_2)\in \C^2: |Im z_1|, |Im z_2|< \mu/2\}$. We follow the elegant proof of Erdogan, Hundertmark, and Lee \cite{EHL} (see  also\cite{HL}) where the corresponding result is proved for dispersion managed solitons in 1D, with the difference that cubic NLS in 2D (from which~\eqref{CR} is derived) is $L^2-$critical as opposed to subcritical. This is reflected in the key estimate~\eqref{refined est} where the ``dimensionless" gain of $\left(\frac{M}{N}\right)^{1/2}$ is much weaker than the $N^{-1/2}$ decay in the 1D case (both gains are dictated by scaling).
\medskip

We start with some notation: Let
$$
F_{\mu, \epsilon}(x)=\mu \frac{|x|}{1+\epsilon |x|}
$$
and define the following weighted versions of the operators $\mathcal E(f_1,f_2,f_3,f_4) $:
\begin{equation}\label{wE}
{\mathcal E}_{\mu, \epsilon}(f_1,f_2,f_3,f_4) =\mathcal E(e^{-F_{\mu, \epsilon}}f_1,e^{-F_{\mu, \epsilon}}f_2,e^{-F_{\mu, \epsilon}}f_3,e^{F_{\mu, \epsilon}}f_4). 
\end{equation}

We shall need the following key estimates on the multilinear operator ${\mathcal E}_{\mu, \epsilon}$.

\begin{prop}\label{weighted est}
Suppose that $f_1, \ldots, f_4\in L^2(\R^2)$, then the following estimates hold uniformly in $\mu, \epsilon$:

\begin{enumerate}[(i)]
\item 
$\displaystyle |{\mathcal E}_{\mu, \epsilon}(f_1,f_2,f_3,f_4)|\lesssim \prod_{j=1}^4\|f_j\|_{L^2}.$
\item Suppose that for some $\ell, k \in \{1,2,3, 4\}$, one has $\operatorname{supp} f_\ell \subset B(0, M)$ and $\operatorname{supp} f_k \subset B(0, N)^c$ with $M \ll N$, then
\begin{equation}\label{refined est}|{\mathcal E}_{\mu, \epsilon}(f_1,f_2,f_3,f_4)|\lesssim \left(\frac{M}{N}\right)^{1/2}\prod_{j=1}^4\|f_j\|_{L^2}.
\end{equation}
\end{enumerate} 
\end{prop}

\begin{proof}
The proof of part $(i)$ will  follow from the bound
\begin{equation}\label{aineq}
-F_{\mu, \epsilon}(\xi+\lambda z)-F_{\mu, \epsilon}(\xi+\lambda z+z^\perp)-F_{\mu, \epsilon}(\xi+z^\perp)+F_{\mu, \epsilon}(\xi)\leq 0 
\end{equation}
for any $\mu, \epsilon$. Notice the elementary inequality: for all $\epsilon>0$, $\eta_1$, $\eta_2$ in $\mathbb{R}^2$,
\begin{equation}\label{bineq}
 \frac{|\eta_{1}+\eta_{2}|}{1+\epsilon|\eta_{1}+\eta_{2}|} \leq \frac{|\eta_{1}\vert +\vert \eta_{2}|}{1+\epsilon(|\eta_{1}\vert +\vert\eta_{2}|)}\leq \frac{|\eta_{1}|}{1+\epsilon |\eta_{1}|}+\frac{|\eta_{2}|}{1+\epsilon |\eta_{2}|}.
\end{equation} 
We get \eqref{aineq} by applying \eqref{bineq} twice, since $\xi=\xi+\lambda z-(\xi+\lambda z+z^\perp)+(\xi+z^{\perp})$.

\medskip

The proof of part $(ii)$ follows by using a refined bilinear Strichartz estimate: Indeed, without loss of generality, we can assume that $f_i \geq 0$, and then
\begin{align*}
|\mathcal E_{\mu, \epsilon}(f_1, f_2, f_3, f_4)|& \leq \mathcal{E}(f_1, f_2, f_3, f_4) \\
& \leq 2\pi \| (e^{it\Delta} f_1) (e^{it\Delta} f_2) (\overline{e^{it\Delta} f_3}) (\overline{e^{it\Delta} f_4}) \|_{L^1_{t,x}} \\ 
& \lesssim \|e^{it\Delta} \widehat f_\ell e^{it\Delta} \widehat f_k\|_{L^2_{t,x}} \|\prod_{j\neq l,k} e^{it\Delta} \widehat f_j\|_{L^2_{t,x}}\\
&\lesssim \left(\frac{M}{N}\right)^{1/2}\|f_\ell\|_{L^2_{x}}\|f_k\|_{L^2_{x}}\prod_{j\neq \ell, k}\|e^{it\Delta}\widehat f_j\|_{L_{t,x}^4}\\
&\lesssim \left(\frac{M}{N}\right)^{1/2}\prod_{j=1}^4\|f_j\|_{L_{x}^2},
\end{align*}
where we used in the third inequality that $\| e^{it\Delta} f e^{it\Delta} g \|_2=\| e^{it\Delta} f \overline{e^{it\Delta} g} \|_2$, in the next to last step the bilinear refinement to the $L^2_{x} \to L^4_{t,x}$ Strichartz estimate in 2D (cf. Bourgain \cite{Bour}) and in the last inequality the standard $L^2_{x} \to L^4_{t,x}$ Strichartz estimate.
\end{proof}

\begin{rema}\label{olivebranch}
Since $\mathcal E (f_1, f_2, f_3, f_4) =\mathcal E (\widehat{f_1}, \widehat{f_2},\widehat{f_3}, \widehat{f_4})$, the same estimates above for $\mathcal E_{\mu, \epsilon}$ hold for the operator
\begin{eqnarray*}
\widetilde{\mathcal E}_{\mu, \epsilon}(f_1, f_2, f_3,f_4)&:=&\mathcal E(e^{-F_{\mu, \epsilon}(P)}f_1,e^{-F_{\mu, \epsilon}(P)}f_2,e^{-F_{\mu, \epsilon}(P)}f_3,e^{F_{\mu, \epsilon}(P)}f_4)\\
&=&\mathcal E_{\mu,\eps} (\widehat{f_1}, \widehat{f_2},\widehat{f_3}, \widehat{f_4})
\end{eqnarray*}
where $P=-i\nabla$, and with the corresponding support assumptions in part $(ii)$ of Proposition~\ref{weighted est} imposed on $\widehat f_\ell$ and $\widehat f_k$.
\end{rema}

\begin{proof}[Proof of Theorem~\ref{analyticity}]

We start with some notation. Assume for simplicity that $\|\phi\|_{L^2}=1$. Let $M\geq2$ to be fixed later, and denote
$$
f_<(x):=f(x)\chi_{|x|\leq M}, \quad f_\sim(x):= f (x)\chi_{M\leq|x|\leq M^2}, \quad f_>(x):=f (x)\chi_{|x|\geq M^2}.
$$
We would like to find $M, \mu>0$ such that $\|e^{F_{\mu,\epsilon}} \phi \|_{L^2} \lesssim 1$ uniformly in $\epsilon>0$. Since $\|e^{F_{\mu,\epsilon}} f_<\|_{L^2}$, $\|e^{F_{\mu,\epsilon}} f_\sim\|_{L^2}\leq e^{\mu M^2} \|f\|_{L^2}$, the hardest part is to bound $\|e^{F_{\mu,\epsilon}} \phi_{>} \|_{L^2}$, which is the aim of the three steps below.

\bigskip

\noindent \underline{Step 1}: Let $\psi:=e^{F_{\mu, \epsilon}} \phi$. We first obtain an estimate on  $\|\psi_{>}\|_{L^2}$. We start by multiplying~\eqref{EL0} by $e^{2F_{\mu, \epsilon}(x)} \chi_{|x|\geq M^2} \phi$, which gives after integration
$$
\omega \|e^{F_{\mu, \epsilon}} \phi_>\|_{L^2}= \mathcal E (\phi, \phi, \phi, e^{2F_{\mu, \epsilon}(x)}\phi_>)=\mathcal E_{\mu, \epsilon}(e^{F_{\mu, \epsilon}} \phi, e^{F_{\mu, \epsilon}} \phi, e^{F_{\mu, \epsilon}} \phi, e^{F_{\mu, \epsilon}} \phi_>). 
$$
Therefore
\begin{align*}
\omega \|\psi_{>}\|_{L^2}^2&=\mathcal E_{\mu, \epsilon}(\psi, \psi, \psi, \psi_>)\\
&=\mathcal E_{\mu, \epsilon}(\psi_<+\psi_\sim+\psi_>,\psi_<+\psi_\sim+\psi_>, \psi_<+\psi_\sim+\psi_>, \psi_>).
\end{align*}

Expanding the terms in this sum, we bound terms that are cubic and quartic in $\psi_>$ using Proposition~\ref{weighted est} by $e^{\mu M^2} \|\psi_>\|^3_{L^2}$ and $\|\psi_>\|^4_{L^2}$ respectively. Terms that involve at least one copy of $\psi_<$ are bounded using part  $(ii)$ of Proposition~\ref{weighted est}. This yields 
\begin{multline*}
\omega \|\psi_{>}\|_{L^2}^2\lesssim \\
\begin{aligned}
&\lesssim \|\psi_>\|^4_{L^2}+e^{\mu M^2} \|\psi_>\|^3_{L^2}+ \left(M^{-1/2} \| \psi_< \|_{L^2}^2 +\|\psi_\sim\|_{L^2}^2 \right)\| \psi_>\|_{L^2}^2+\left(M^{-1/2} \| \psi_< \|_{L^2}^3 +\|\psi_\sim\|_{L^2}^3\right)\|\psi_>\|_{L^2}\\
&\lesssim \|\psi_>\|^4_{L^2}+e^{\mu M^2} \|\psi_>\|^3_{L^2}+ e^{2\mu M^2}\left(M^{-1/2}  +\|\phi_\sim\|_{L^2}^2 \right)\| \psi_>\|_{L^2}^2+e^{3\mu M^2}\left(M^{-1/2}  +\|\phi_\sim\|_{L^2}^3\right)\|\psi_>\|_{L^2}.
\end{aligned} 
\end{multline*}
Dividing both sides by $\|\psi_{>}\|_{L^2}$ we obtain 
\begin{multline}\label{psi est}
\omega \|\psi_{>}\|_{L^2}\leq \\
C \left( \|\psi_>\|^3_{L^2}+e^{\mu M^2} \|\psi_>\|^2_{L^2}+ e^{2\mu M^2}\left(M^{-1/2}  +\|\phi_\sim\|_{L^2}^2 \right)\| \psi_>\|_{L^2}+e^{3\mu M^2}\left(M^{-1/2}  +\|\phi_\sim\|_{L^2}^3\right)\right).
\end{multline}

\bigskip

\noindent \underline{Step 2:} Let $\nu:=\|\psi_>\|_{L^2}$ and choose $\mu=M^{-2}$. Estimate~\eqref{psi est} above translates into 
$$
\left(\omega - C_0 M^{-1/2}-C_0\|\phi_\sim\|_{L^2}^2\right)\nu- C_0\nu^2-C_0\nu^3\leq C_0 \left(M^{-1/2}  +\|\phi_\sim\|_{L^2}^3\right).
$$
for a new constant $C_0>0$

Let $G(\nu)=\frac{\omega}{2}\nu -C_0 \nu^2-C_0 \nu^3$ and denote by $\nu_{max}$ the point at which $\nu \mapsto G(\nu)$ achieves its maximum on $[0, \infty)$. Let $\nu_0=\nu_{max}/2$ and choose $M$ large enough so that the following two conditions are satisfied:

\begin{enumerate}
\item $C_0 (M^{-1/2}+\|\phi_\sim\|_{L^2})\leq \min (\omega/2, G(\nu_0))$,
\item $\|\phi_{\sim}\|_{L^2}+\|\phi_>\|_{L^2} \leq \frac{\nu_0}{2}$.
\end{enumerate}
This is possible by the monotone convergence theorem.

Therefore,~\eqref{psi est} gives
$$
G(\|e^{F_{\mu,\epsilon}}\phi_>\|_{L^2})\leq G(\nu_0)
$$
for all $\epsilon>0$. 

\bigskip

\noindent \underline{Step 3:} Since $\|e^{F_{\mu,\epsilon}} \phi_>\|_{L^2}$ is continuous in $\epsilon$, we obtain that $\{\|e^{F_{\mu,\epsilon}} \phi_>\|_{L^2}: \epsilon >0\}$ is contained in only one of the two connected components of $G^{-1}(-\infty, G(\nu_0))$. Now
$$
\|e^{F_{\mu,1}} \phi_>\|_{L^2}\leq \|e^{\mu\frac{|x|}{1+|x|}}\|_{L^\infty}\|\phi_>\|_{L^2}\leq 2 \|\phi_>\|_{L^2}\leq \nu_0,
$$
(using our choice  $\mu=M^{-2}\leq 1/4$ above), which gives that $\{\|e^{F_{\mu,\epsilon}} \phi_>\|_{L^2}: \epsilon >0\}$ is contained in the component $[0, \nu_0]$, which means that 

$$
\|e^{F_{\mu,\epsilon}} \phi_>\|_{L^2} \leq \nu_0,\qquad \text{for all } \epsilon>0.
$$
\medskip

We can now complete the proof. Taking the limit $\epsilon\to 0$, we obtain by monotone convergence that $\|e^{\mu|x|}\phi_>\|_{L^2} \leq \nu_0$, which implies in turn $\|e^{\mu|x|}\phi\|_{L^2} < \infty$.
A parallel argument using Remark~\ref{olivebranch} gives that there exists $\mu'$ such that $e^{\mu' |\xi|}\widehat \phi\ \in L^2$. The result now follows easily.
\end{proof}

\section{\texorpdfstring{Studying  \eqref{CR} in the basis of special Hermite functions}{Studying (CR) in the basis of special Hermite functions}} 

\label{sectionSHF}

We are able to describe quite precisely the non-linear operator $\T$, in the Hilbertian basis of $L^2(\R^2)$ provided by the special Hermite functions, which will be defined below.

\subsection{The special Hermite functions}
We first recall  some elements of the spectral theory of $H$ and\;$L$. We follow mainly \cite[Appendix~D]{Cohen}, see also~\cite{Thanga}. Define the creation and annihilation operators
\begin{equation*}
a_{x}\overset{def}{=}\frac{1}{\sqrt{2}}(x+\partial_x),\quad a^{\star}_{x}=\frac{1}{\sqrt{2}}(x-\partial_x) \;\;\mbox{  and  }\;\;
a_{y}\overset{def}{=}\frac{1}{\sqrt{2}}(y+\partial_y),\quad a^{\star}_{y}=\frac{1}{\sqrt{2}}(y-\partial_y).
\end{equation*}
In the sequel, it will be more convenient to work in complex coordinates. Therefore we set $z=x+iy$, $\overline{z}=x-iy$ and $\partial_{z}=\frac12(\partial_{x}-i\partial_{y})$, $\partial_{\overline{z}}=\frac12(\partial_{x}+i\partial_{y})$ before defining
\begin{equation*}
a_{d}\overset{def}{=}\frac{1}{\sqrt{2}}(a_x-ia_y)=\frac12(\ov{z}+2\partial_{z}),\quad a_{g}\overset{def}{=}\frac{1}{\sqrt{2}}(a_x+ia_y)=\frac12({z}+2\partial_{\ov{z}}).
\end{equation*}
We record the following formulas
\begin{equation}
\label{agd}
\begin{split}
& a^{\star}_{d}=\frac12({z}-2\partial_{\ov{z}}),\quad a^{\star}_{g}=\frac12(\ov{z}-2\partial_{{z}}) \\
& [a_d, a_d^\star] = [a_g, a_g^\star] = 1, \quad [a_g,a_d^\star] = [a_g,a_d] = 0 \\
& \F(a^{\star}_d u)=-i a^{\star}_d \wh{u}, \quad \F(a^{\star}_g u)=-i a^{\star}_g \wh{u}\\
& H= -4\partial_{{z}} \partial_{\ov{z}}+\vert z\vert^2   = 2\big(a^{\star}_{d}\,a_d+a^{\star}_{g}\,a_g+1\big) \\
& L= z\partial_{{z}}- \ov{z}\partial_{\ov{z}}=a^{\star}_{d}\,a_d-a^{\star}_{g}\,a_g \\
& x\cdot \nabla= z\partial_{{z}}+\ov{z}\partial_{\ov{z}}.
\end{split}
\end{equation}
We are now able to define the so-called special Hermite functions
\begin{equation*}
\psi_{n,m}=\frac1{\sqrt{\pi n!\, m!}}\big(a^{\star}_d\big)^n \big(a^{\star}_g\big)^m \e^{-z\ov{z}/2},
\end{equation*}
and if $n+m$ is even we set
\begin{equation*}
\phi_{n,m}=\psi_{\frac{n+m}2,\frac{n-m}2}.
\end{equation*}
It is easy to show that 
\begin{equation*}
a_g \psi_{n,m}=\sqrt{m}  \psi_{n,m-1},\qquad a_d \psi_{n,m}=\sqrt{n}  \psi_{n-1,m},
\end{equation*}
\begin{equation*} 
a^{\star}_g \psi_{n,m}=\sqrt{m+1}  \psi_{n,m+1},\qquad a^{\star}_d \psi_{n,m}=\sqrt{n+1}  \psi_{n+1,m},
\end{equation*}
which implies
\begin{equation}\label{a3}
a_g \phi_{n,m}=\sqrt{\frac{n-m}2}\phi_{n-1,m+1},\qquad a_d \phi_{n,m}=\sqrt{\frac{n+m}2} \phi_{n-1,m-1},
\end{equation}
\begin{equation}\label{a4}
a^{\star}_g \phi_{n,m}=\sqrt{\frac{n-m+2}2}  \phi_{n+1,m-1},\qquad a^{\star}_d \phi_{n,m}=\sqrt{\frac{n+m+2}2}  \phi_{n+1,m+1}.
\end{equation}

It turns out that the families  $(\psi_{n,m}, n\geq 0,m\geq 0 )$ and   $(\phi_{n,m}, n\geq 0,m\geq 0 )$ are well-adapted to\;$H$,\;$L$ and $\mathcal{F}$.

\begin{prop}
\begin{enumerate}[(i)]
\item The family $(\psi_{n,m}, n\geq 0,m\geq 0 )$ is an $L^2$-normalised Hilbertian basis of the space $L^{2}(\R^2)$ such that 
\begin{equation*} 
H \psi_{n,m}=2(n+m+1) \psi_{n,m},\qquad L \psi_{n,m}=(n-m) \psi_{n,m}, \qquad \wh{\psi}_{n,m}=\e^{-i(n+m)\frac{\pi}2}\psi_{n,m}.
\end{equation*}
\item The family $(\phi_{n,m}, n\geq 0,-n\leq m\leq n, n+m\;\text{even} )$ is an $L^2$-normalised Hilbertian basis of $L^{2}(\R^2)$ such that 
\begin{equation*}
H \phi_{n,m}=2(n+1) \phi_{n,m},\qquad L \phi_{n,m}=m \phi_{n,m}, \qquad \wh{\phi}_{n,m}=\e^{-in\frac{\pi}2}\phi_{n,m}.
\end{equation*}
\end{enumerate}
\end{prop}
\begin{proof}
It follows in a straightforward way from the formulas~(\ref{agd}).
\end{proof}
In other words, $(\psi_{n,m})$ and $(\phi_{n,m})$ are Hilbertian bases of common eigenfunctions of $H$, $L$ and $\mathcal{F}$, which is consistent with the fact that these linear operators commute with one another. 

 Define the eigenspace $E_n=span\{ \phi_{n,m},   \;-n\leq m\leq n, n+m \;\text{even}\}$. Let $u\in L^2(\R^2)$ which can be written 
\begin{equation}\label{expansionf}
u=\sum_{n=0}^{+\infty}u_n,\quad \text{with}\quad  u_n= \sum_{m=-n}^{n}c_{n,m}\phi_{n,m}\in E_n,
\end{equation}
with the convention $c_{n,m}=0$ if $n+m$ is odd.

\subsection{\texorpdfstring{The operator $\mathcal{T}$ in the basis of special Hermite functions}{The operator  T in the basis of special Hermite functions}}

\begin{prop} \ph \label{propFormule} We have
\begin{equation}\label{ha}
\mathcal E(\phi_{n_1,m_1},\phi_{n_2,m_2},\phi_{n_3,m_3},\phi_{n_4,m_4})=\pi^2 \left(\int_{\R^2} \phi_{n_1,m_1}\phi_{n_2,m_2}\ov{\phi_{n_3,m_3}\phi_{n_4,m_4}}dx\right)\mathbbm{1}_{\big\{\substack {n_1+n_2=n_3+n_4\\m_1+m_2=m_3+m_4}\big\}},
\end{equation}
and 
\begin{equation} \label{coli}
\T(\phi_{n_1,m_1},\phi_{n_2,m_2},\phi_{n_3,m_3})= \pi^2 \left(\int_{\R^2} \phi_{n_1,m_1}\phi_{n_2,m_2}\ov{\phi_{n_3,m_3}\phi_{n_4,m_4}}dx\right) \phi_{n_4,m_4}
\end{equation}
with $n_4 = n_1 + n_2 - n_3$ and $m_4 = m_1 + m_2 - m_3$
(actually, $\dis \int_{\R^2} \phi_{n_1,m_1}\phi_{n_2,m_2}\ov{\phi_{n_3,m_3}\phi_{n_4,m_4}}dx=0$ if $m_1+m_2\neq m_3+m_4$).
\end{prop}

\begin{proof} We first check that
\begin{equation*}
\T(\phi_{n_1,m_1},\phi_{n_2,m_2},\phi_{n_3,m_3})=\mathcal E(\phi_{n_1,m_1},\phi_{n_2,m_2},\phi_{n_3,m_3},\phi_{n_4,m_4}) \phi_{n_4,m_4}
\end{equation*}
with $n_4 = n_1 + n_2 -n_3$ and $m_4 = m_1 + m_2 - m_3$.
Since $\phi_{n,m}$ is a common eigenfunction of $H$ and\;$L$ with eigenvalues respectively $n$ and $m$, Corollary~\ref{loriot} implies that $\T(\phi_{n_1,m_1},\phi_{n_2,m_2},\phi_{n_3,m_3})$ is an eigenfunction of $H$ and $L$ with eigenvalues respectively $n_4$ and $m_4$. Thus it is collinear to $\phi_{n_4,m_4}$. The definition of $\mathcal{E}$ gives the desired result.

Next by \eqref{eq2}, 
\begin{equation*}
\mathcal E(\phi_{n_1,m_1},\phi_{n_2,m_2},\phi_{n_3,m_3},\phi_{n_4,m_4}) = 2\pi \int_{-\frac{\pi}4}^{\frac{\pi}4}\e^{-2i(n_1+n_2-n_3-n_4)t}dt  \int_{\R^2} \phi_{n_1,m_1}\phi_{n_2,m_2}\ov{\phi_{n_3,m_3}\phi_{n_4,m_4}}dx,
\end{equation*}
therefore if $n_1+n_2=n_3+n_4$ and $m_1+m_2=m_3+m_4$
\begin{equation*}
\mathcal E(\phi_{n_1,m_1},\phi_{n_2,m_2},\phi_{n_3,m_3},\phi_{n_4,m_4})= \pi^2 \int_{\R^2} \phi_{n_1,m_1}\phi_{n_2,m_2}\ov{\phi_{n_3,m_3}\phi_{n_4,m_4}}dx,
\end{equation*}
which concludes the proof.
\end{proof}

As a result, expanding $f$ as in~(\ref{expansionf}), the equation \eqref{CR} is equivalent to
 \begin{equation*}
 i\dot{u}_n= \sum_{\substack {n_1,n_2,n_3\geq 0\\n_1+n_2-n_3=n}}\T({u}_{n_1},{u}_{n_2},{u}_{n_3}),
 \end{equation*}
 or: for all $n\geq0$ and $-m\leq m\leq n$
  \begin{equation}\label{eqmod}
 i\dot{c}_{n,m}=  \pi^2 \sum_{\substack {n_1,n_2,n_3\geq 0\\n_1+n_2-n_3=n}} \,    \sum_{\substack {-n_j\leq m_j\leq n_j \\m_1+m_2-m_3=m}} \left(\int_{\R^2} \phi_{n_1,m_1}\phi_{n_2,m_2}\ov{\phi_{n_3,m_3}\phi_{n,m}}dx\right)  c_{n_1,m_1}c_{n_2,m_2}\ov{c_{n_3,m_3}}.
 \end{equation}

\subsection{Conservation laws in the basis of special Hermite functions}

 \begin{lemm}\ph Let $Q$ be one of the operators $a_d$, $a_g$,  $a^{\star}_d$ or $a^{\star}_g$. Then for all $f_1,f_2,f_3\in \mathcal{S}(\R^2)$
  \begin{equation*}
Q\big(\T(f_1,f_2,f_3)\big)= \T(Qf_1,f_2,f_3) +\T(f_1,Qf_2,f_3) -\T(f_1,f_2,Q^{\star}f_3). 
\end{equation*}
As a consequence, 
\begin{equation*}
  \int_{\R^2} (a_d u) \ov{u},\quad \int_{\R^2} (a_g u) \ov{u},\quad \int_{\R^2} (a^{\star}_d u) \ov{u},\quad \int_{\R^2} (a^{\star}_g u) \ov{u}
\end{equation*}
are conservation laws for \eqref{CR}.
   \end{lemm}

\begin{prop} Set $u=\sum_{n=0}^{+\infty}\sum_{m=-n}^{n}c_{n,m} \phi_{n,m}$ with the convention that $c_{n,m}=0$ if $m+n$ is odd. In these coordinates, the conservation laws are the following
\begin{equation*}
\Vert u\Vert^2_{L^2(\R^2)}=\sum_{n=0}^{+\infty}\sum_{m=-n}^{n}\vert c_{n,m}\vert^2
\end{equation*}
\begin{equation*}
\Vert H u\Vert^2_{L^2(\R^2)}=2\sum_{n=0}^{+\infty}(n+1)\sum_{m=-n}^{n}\vert c_{n,m}\vert^2
\end{equation*}
\begin{equation*}
\int_{\R^2} (L u)\ov{u}=\sum_{n=0}^{+\infty}\sum_{m=-n}^{n}m\vert c_{n,m}\vert^2
\end{equation*}
\begin{equation}\label{k5}
\int_{\R^2} (a_d u)\ov{u}=\int_{\R^2} u\,\ov{a^{\star}_d u}=\sum_{n=0}^{+\infty}\sum_{m=-n}^{n} \sqrt{\frac{n+m+2}2} c_{n+1,m+1}\,\ov{c_{n,m}} 
\end{equation}
\begin{equation}\label{k6}
\int_{\R^2} (a_g u)\ov{u}=\int_{\R^2} u\,\ov{a^{\star}_g u}=\sum_{n=0}^{+\infty}\sum_{m=-n}^{n} \sqrt{\frac{n-m+2}2} c_{n+1,m-1}\,\ov{c_{n,m}} 
\end{equation}
\begin{equation}\label{k7}
\int_{\R^2} \vert z\vert^2 \vert u\vert^2=\sum_{n=0}^{+\infty}(n+1)\sum_{m=-n}^{n}\vert c_{n,m}\vert^2+\sum_{n=0}^{+\infty}\sum_{m=-n}^{n} \frac{\sqrt{n^2-m^2}}2\big( c_{n,m}\,\ov{c_{n-2,m}}+\ov{c_{n,m}}\,{c_{n-2,m}} \big)
\end{equation}
\begin{equation}\label{k8}
\int_{\R^2} \big(x\cdot \nabla u \big) \ov{u}=-\sum_{n=0}^{+\infty}\sum_{m=-n}^{n}\vert c_{n,m}\vert^2+\sum_{n=0}^{+\infty}\sum_{m=-n}^{n} \frac{\sqrt{n^2-m^2}}2\big( c_{n,m}\,\ov{c_{n-2,m}}-\ov{c_{n,m}}\,{c_{n-2,m}} \big).
\end{equation}
\end{prop}

\begin{proof} The three first relations are straightforward. For\;\eqref{k5} and \eqref{k6} we use \eqref{a3} and \eqref{a4}. For\;\eqref{k7}, we   write $z=a_g+a^{\star}_d$ and $\ov{z}=a_d+a^{\star}_g$ to get
\begin{equation*}
z\phi_{n,m}=\sqrt{\frac{n-m}2}\phi_{n-1,m+1}+\sqrt{\frac{n+m+2}2}  \phi_{n+1,m+1}
\end{equation*}
and 
\begin{equation*}
\vert z\vert^2 \phi_{n,m}=(n+1) \phi_{n,m}+\frac{\sqrt{n^2-m^2}}2  \phi_{n-2,m}+\frac{\sqrt{(n+2)^2-m^2}}2  \phi_{n+2,m}.
\end{equation*}
We now turn to \eqref{k8}. Similarly,
\begin{equation*}
\partial_z\phi_{n,m}=\frac12\sqrt{\frac{n+m}2}\phi_{n-1,m-1}-\frac12\sqrt{\frac{n-m+2}2}  \phi_{n+1,m-1},
\end{equation*}
and 
\begin{equation*}
z\partial_z\phi_{n,m}=\frac12(m-1)\phi_{n,m}+\frac14\sqrt{n^2-m^2}\phi_{n-2,m}-\frac14\sqrt{(n+2)^2-m^2}\phi_{n+2,m},
\end{equation*}
and using that $\ov{\phi_{n,m}}={\phi_{n,-m}}$ we get 
\begin{equation*}
\ov{z}\partial_{\ov{z}}\phi_{n,m}=\frac12(-m-1)\phi_{n,m}+\frac14\sqrt{n^2-m^2}\phi_{n-2,m}-\frac14\sqrt{(n+2)^2-m^2}\phi_{n+2,m},
\end{equation*}
which gives the desired result.
\end{proof}


\section{Dynamics on the eigenspaces of \texorpdfstring{$H$ }{H}}

\label{sectioneigenspacesH}

Recall that, for $N\in \mathbb{N}$, $E_N$ is the $N$-th eigenspace of $-\Delta+|x|^2$, associated to the eigenvalue $2N + 2$. It is spanned by the eigenfunctions $\phi_{N,m}$, $m \in I_N = \{ -N,-N+2,\dots,N-2,N \}$, thus natural coordinates on $E_N$ are provided by the $(c_{N,m})$ (which we simply denote $(c_m)$ when the context is clear): if $u \in E_N$,
$$
u = \sum_{m\in I_N} c_m \phi_{N,m}.
$$ 
There are only two independent conserved quantities for the restriction of  \eqref{CR} to $E_N$:
$$
M(u) = \|u\|_{L^2}^2 = \sum_{m \in I_N} |c_{m}|^2 \quad  \mbox{and} \quad P(u) = \int L u \cdot \ov{u} = \sum_{m \in I_N} m |c_{m}|^2,
$$ 
the associated symmetries being of course phase and space rotation: $u \mapsto e^{i\theta} u$ and $u \mapsto R_\lambda u$.

\subsection{\texorpdfstring{Dynamics on $E_0$}{Dynamics on E0}}\label{sect41} The eigenspace $E_0$ is generated by the Gaussian $\phi_{0,0}(x)=\frac{1}{\sqrt{\pi}}e^{-\frac{1}{2}|x|^2}$. For data $u(t=0)=c_0 \phi_{0,0}$, the solution $u(t)=c(t)\phi_{0,0}$ is given by $c(t) = e^{-i\frac{\pi}{2}|c_0|^2 t} c_0$.
 
\subsection{\texorpdfstring{Dynamics on $E_1$}{Dynamics on E1}} Write $u=c_1  \phi_{1,1}+ c_{-1}  \phi_{1,-1}$. By Lemma \ref{lem11},
$$
\mathcal{E}(u) = \pi^2 \int_{\mathbb{R}^2} |u|^4.
$$
Using that $\int \vert \phi_{1,1} \vert^4=\int \vert \phi_{1,-1}\vert^4=1/(4\pi)$, it is easy to see that the Hamiltonian reduces on $E_1$ to
\begin{eqnarray*}
\mathcal E(u)&=&\pi^2 \int_{\R^2}\vert  c_1  \varphi_{1,1}+ c_{-1}  \varphi_{1,-1}  \vert^4\\ 
&=&\frac{\pi}{4}\big(\vert c_1\vert^4+\vert c_{-1}\vert^4 +4\vert c_1\vert^2\vert c_{-1}\vert^2   \big).
\end{eqnarray*}
The equation \eqref{CR} can be written $i\partial_t c = \frac{1}{2}\frac{\partial \mathcal{E}}{\partial \bar c}$ or in other words
\begin{equation*} 
\left\{
\begin{aligned}
&i\dot{c}_1=\frac{\pi}{4}\big(\vert c_1\vert^2 +  2\vert c_{-1}\vert^2\big) c_1 , \\
&i\dot{c}_{-1}=\frac{\pi}{4}\big(2\vert c_1\vert^2 +  \vert c_{-1}\vert^2\big) c_{-1}.
\end{aligned}
\right.
\end{equation*}
It is now easy to integrate this equation: if $u_0=c^0_1 \phi_{1,1}+c^0_{-1} \phi_{1,-1}$,
$$
u(t)=c^0_1 \exp\big(-\frac{i\pi t}{4}(\vert c^0_1\vert^2 +  2\vert c^0_{-1}\vert^2) \big) \phi_{1,1}+ c^0_{-1} \exp\big(-\frac{i\pi t}{4}(2\vert c^0_1\vert^2 +  \vert c^0_{-1}\vert^2) \big)  \phi_{1,-1},
$$
thus in particular, every solution is quasi-periodic; even more, every solution is an $M+\alpha P$ wave, for some $\alpha$.

Even if the variational structure on $E_1$ is fairly simple, let us record it before moving on to the more complicated situation on $E_2$. Maximizers of $\mathcal{E}$ for fixed mass are the $\{c_i\}$ such that $|c_1|=|c_{-1}|$; and minimizers of $\mathcal{E}$ for fixed $M$ satisfy $c_1 = 0$ or $c_{-1} =0$. These give rise to $M$-stationary waves. All the other solutions are non-trivial $M+\alpha P$ waves, which are maximizers of $\mathcal{E}$ for $M+\alpha P$ fixed, with $|\alpha|<\frac{1}{3}$. If $1\neq |\alpha|>\frac{1}{3}$, the maximizers degenerate and are given by $c_1=0$ or $c_{-1} =0$, whereas if $\alpha=1, -1$ the constraint $M+\alpha P$ degenerates and the maximizers can be infinite.

Thus, all the solutions can be obtained as extremizers, and are orbitally stable, in the sense that the moduli $|c_1(t)|$, $|c_{-1}(t)|$ are stable with respect to perturbations of the data, uniformly in time, but angles are not. However, this can be seen directly, without resorting to variational considerations!

\subsection{\texorpdfstring{Dynamics on $E_2$}{Dynamics on E2}} We have a good picture of the dynamics in the eigenspace $E_{2}$.
\subsubsection{Writing down the equation}
Decomposing $u = c_2 \phi_{2,2} + c_0 \phi_{2,0} + c_{-2} \phi_{2,-2}$, we find as above that
\begin{align*}
\mathcal{E}(u) & = \pi^2 \int_{\R^2} |u|^4 \\
& =  \frac{\pi}{4} \left[ \frac{3}{4}|c_2|^4 + \frac{3}{4}|c_{-2}|^4 + |c_0|^4 + 3 |c_2|^2 |c_{-2}|^2 + 2 |c_{-2}|^2 |c_0|^2 + 2 |c_2|^2 |c_0|^2 + c_2 c_{-2} \overline{c_0}^2 + \overline{c_{-2} c_2} c_3^2 \right],
\end{align*}
while the conserved quantities of  \eqref{CR} read
\begin{align*}
& M(u) = |c_{-2}|^2 + |c_0|^2 + |c_2|^2 \\
& P(u) = |c_2|^2 - |c_{-2}|^2.
\end{align*}
One can check that the functionals $\mathcal{E}$, $M$, $P$ are in involution on $E_2$, which has dimension 6; this makes the 
system completely integrable. Using the formulation $i \dot{c} = \frac{1}{2} \frac{\partial}{\partial \bar c} \mathcal{E}(c)$ , the equation reads
\begin{equation*}
\left\{
\begin{aligned}
& i \dot{c_2} = \frac{\pi}{16} \left[ 3 |c_2|^2 c_2 + 6 |c_{-2}|^2 c_2 + 4 |c_0|^2 c_2 + 2 \ov{c_{-2}} c_0^2 \right] \\
& i \dot{c_{-2}} = \frac{\pi}{16} \left[ 3 |c_{-2}|^2 c_{-2} + 6 |c_2|^2 c_{-2} + 4 |c_0|^2 c_{-2} + 2 \ov{c_2} c_0^2 \right] \\
& i \dot{c_0} =\frac{\pi}{16} \left[ 4 |c_0|^2 c_0 + 4 |c_2|^2 c_0 + 4 |c_{-2}|^2 c_0 + 4 c_2 c_{-2} \ov{c_0} \right].
\end{aligned}
\right.
\end{equation*}
Rescaling time by $\tau = 16 \pi t$ and switching to the unknown function $d_i(\tau) = e^{-4iM\tau} c_i(\tau)$, this becomes
\begin{equation}\label{eqd}
\left\{
\begin{aligned}
&  i \dot{d_2} = - |d_2|^2 d_2 + 2 |d_{-2}|^2 d_2 + 2 \ov{d_{-2}} d_0^2 \\
& i \dot{d_{-2}} = - |d_{-2}|^2 d_{-2} + 2 |d_2|^2 d_{-2} + 2 \ov{d_2} d_0^2 \\
& i \dot{d_0} = 4 d_2 d_{-2} \ov{d_0}.
\end{aligned}
\right.
\end{equation}

\subsubsection{The $M$- and $M+\alpha P$- stationary waves}
A computation gives all the $M$- and $M+\alpha P$-stationary waves:
\begin{itemize}
\item[(a)] $(d_2,d_{-2},d_0) = (z,0,0) e^{i\mu t}$ with $z \in \mathbb{C}$, and $\mu = |z|^2$. 
\item[(b)] $(d_2,d_{-2},d_0) = (0,z,0) e^{i\mu t}$ with $z \in \mathbb{C}$, and $\mu =|z|^2$.
\item[(c)] $(d_2,d_{-2},d_0) = (0,0,z )$ with $z \in \mathbb{C}$.
\item[(d)] $(d_2,d_{-2},d_0) = (z,z',0) e^{i\mu t}$ with $z,z' \in \mathbb{C}$, $|z| = |z'|$, and $\mu = -|z|^2$.
\item[(e)] $(d_2,d_{-2},d_0) = \lambda \left( \sqrt{\frac{2}{9}} e^{i\beta_1}, \sqrt{\frac{2}{9}} e^{i\beta_2}, \pm i \sqrt{\frac{5}{9}} e^{i \frac{\beta_1 + \beta_2}{2}} \right) e^{i\mu t}$ with $\lambda,\beta_1,\beta_2 \in \mathbb{R}$, and $\mu = \frac{8}{9}\lambda^2$. 
\item[(f)] $(d_2,d_{-2},d_0) = \lambda \left( \sqrt{\frac{2}{7}} e^{i\beta_1}, \sqrt{\frac{2}{7}} e^{i\beta_2}, \pm \sqrt{\frac{3}{7}} e^{i \frac{\beta_1+ \beta_2}{2}} \right) e^{i\mu t}$ with $\lambda,\beta_1,\beta_2 \in \mathbb{R}$, and $\mu = -\frac{8}{7} \lambda^2$.
\item[(g)] $(d_2, d_{-2}, d_0)= R_{-\alpha \omega t} (z, z', 0)e^{i\omega t}$ with $|z|\neq |z'|$ non-zero, $\omega=\frac{|z'|^2}{1-\gamma}$, $\gamma= \frac{|z|^2-|z'|^2}{|z|^2+|z'|^2}$, and $\alpha= \gamma \frac{|z|^2+|z'|^2-\omega}{2 \omega}$. 
\item[(h)] $(d_2,d_{-2},d_0) =  \left( x e^{i \beta_1} e^{i(\mu+\nu)t} , y e^{i \beta_2} e^{i(\mu-\nu)t} , \epsilon z e^{i \frac{\beta_1+\beta_2}{2}} e^{i \mu t} \right)$, with $x,y,z,\beta_1,\beta_2,\mu,\nu \in \mathbb{R}$, $\mu \in (-\infty,-\frac{8}{3}z^2] \cup [\frac{8}{5}z^2,8z^2)$,
$$
\left\{ 
\begin{array}{l} 
x^2 = \mu \big( \frac{\mu}{8z^2 - \mu} - \frac{\nu}{8z^2 - 3 \mu} \big) \\[5pt]
y^2 = \mu \big( \frac{\mu}{8z^2 - \mu} + \frac{\nu}{8z^2 - 3 \mu} \big) \\[5pt]
\nu^2 = \left(8 z^2 - 3 \mu \right)^2 \big( -\frac{1}{16} + \frac{\mu^2}{(8z^2 - \mu)^2} \big),
\end{array}
\right.
$$
and finally $\epsilon = \pm 1$ if $\mu<0$, $\epsilon = \pm i$ if $\mu>0$.
\end{itemize}
Before discussing the stability of these solutions, let us show how the system can be integrated. 

\subsubsection{Integrating the equation}
Setting (recall that $P=|d_2|^2-|d_{-2}|^2$)
$$
D_0=d_0, \quad D_2 = e^{-itP} d_2 \quad \mbox{and} \quad D_{-2} = e^{itP} d_{-2},
$$
the equation~\eqref{eqd} becomes
\begin{equation*} 
\left\{
\begin{aligned}
& i \dot{D_2} =  |D_{-2}|^2 D_2 + 2 \ov{D_{-2}} D_0^2 \\
& i \dot{D}_{-2} = |D_2|^2 D_{-2} + 2 \ov{D_2} D_0^2 \\
& i \dot{D_0} = 4 D_2 D_{-2} \ov{D_0}.
\end{aligned}
\right.
\end{equation*}
Setting now $D_2 = A e^{i\alpha}$, $D_{-2} = B e^{i\beta}$, $D_0 = C e^{i\gamma}$, the conservation of mass and momentum give the relations
$$
A = \sqrt{\frac{M+P-C^2}{2}}, \quad B = \sqrt{\frac{M-P-C^2}{2}}
$$
which allow to eliminate $A$ and $B$ and obtain the new equation
$$
\left\{ \begin{array}{l} \dot C = 4ABC \sin(\alpha + \beta -2\gamma) \\[5pt]
 \dot{\alpha} = -B^2 - 2 \frac{BC^2}{A} \cos(\alpha + \beta - 2 \gamma) \\[5pt]
 \dot{\beta} = -A^2 - 2 \frac{AC^2}{B} \cos(\alpha + \beta - 2 \gamma) \\[5pt]
 \dot{\gamma} = -4AB \cos (\alpha + \beta - 2\gamma). \end{array} \right. 
$$
Switching to the new unknown $\xi = \alpha + \beta - 2\gamma$, we obtain the two-dimensional ODE
\begin{equation*} 
\left\{ \begin{array}{l}
 \dot C = 4 ABC \sin \xi \\[5pt]
 \dot \xi = -A^2 - B^2 - 2\left( \frac{BC^2}{A} + \frac{AC^2}{B} - 4 AB \right) \cos \xi.
 \end{array} 
\right.
\end{equation*}
As a two dimensional ODE, it can be fully understood (by plotting the phase portrait) and then we can deduce the behavior of  the full system. We do not pursue this direction here.

\subsubsection{Orbital stability of the $M$- and $M+\alpha P$-stationary waves.} It can be deduced from variational considerations, as well as the reduction to a two-dimensional ODE that was just presented. We examine one by one the waves presented above.
\begin{itemize}
\item[(a)] is orbitally stable since it maximizes the angular momentum for fixed mass.
\item[(b)] is orbitally stable since it minimizes the angular momentum for fixed mass.
\item[(c)] and (d) are not stable, and there are actually orbits joining arbitrarily small neighborhood of the former and the latter. This can be most easily seen by considering the reduced system in $(C,\xi)$ in the case $P=0$. Assume that $M=1$, then $\left\{ \begin{array}{l} \dot C = 2C(1-C^2) \sin \xi \\ \dot \xi =(C^2-1) + 4(1-2C^2) \cos \xi \end{array} \right.$. The waves under consideration correspond to the orbits $\left\{ \begin{array}{l} C=0 \\ \dot \xi = -1 + 4 \cos \xi \end{array} \right.$ and $\left\{ \begin{array}{l} C=1 \\ \dot \xi = - 4 \cos \xi \end{array} \right.$, and an analysis of the phase diagram gives the desired conclusion.
\item[(e)] is orbitally stable since it minimizes the Hamiltonian for fixed mass, as a lengthy computation shows.
\item[(f)] is orbitally stable since it maximizes the Hamiltonian for fixed mass, as a lengthy computation shows.
\item[(h)] is orbitally stable since it maximizes (for $\mu<0$) or minimizes (for $\mu>0$) the Hamiltonian for $M+\alpha P$ fixed - for a properly chosen $\alpha$.
\end{itemize}

\subsection{\texorpdfstring{Dynamics on $E_N$}{Dynamics on EN}} \label{sect44}

Some  $M$-stationary waves on $E_N$ are given by all the $\phi_{N,k}$, with $k \in I_N$. The waves $\phi_{N,N}$ and $\phi_{N,-N}$ are orbitally stable as extremizers of $P$ for $M$ fixed. By analogy with the case $N=2$, it seems natural to expect that the other $\phi_{N,k}$ are unstable.

These are by no means the only stationary waves: for instance, it is easy to check that $z \phi_{N,k} + z' \phi_{N,-k}$ gives rise to a stationary wave if $|z|=|z'|$, and $k,k' > \frac{N}{2}$. Still by analogy with the case $N=2$, this wave should also be unstable.

Other orbitally stable waves should be obtained by minimization, or maximization of $\mathcal{E}$ for $M$, or $M+\alpha P$ fixed. It seems plausible that this extremization procedure should produce new waves than the ones which have already been described. Focusing on the case where the mass is fixed: $M=1$, this would be the case if 
\begin{equation}
\label{minmax}
\max_{M(\phi)=1} \mathcal{E}(\phi) > \max_{k\in I_N}\mathcal{E} (\phi_{N,k}), \quad \mbox{respectively} \quad \min_{M(\phi)=1} \mathcal{E}(\phi) < \min_{k\in I_N} \mathcal{E}(\phi_{N,k}).
\end{equation}
This does not follow from known estimates: we only know that
$$
\max_{M(\phi)=1} \mathcal{E}(\phi) \approx N^{-\frac{1}{3}}, \quad \mathcal{E}(\phi_{N,0}) \lesssim N^{-1} (\log N),\quad \mbox{and} \quad \mathcal{E}(\phi_{N,N}) =   \mathcal{E}(\phi_{N,-N}) \approx N^{-\frac{1}{2}}
$$
(the first estimate is taken from~\cite{KT}, the second one, valid for $N$ even only, from~\cite{IRT}, and the last one is a simple computation). However, there are good reasons to believe that, for instance, the first inequality in~\eqref{minmax} should hold. Indeed, the near maximizers of $\mathcal{E}(\phi)$, as explained in~\cite{KT}, are expected to focus along rays, or at points. This is not possible for the $\phi_{N,k}$, which satisfy $|\phi_{N,k}(z)|  = \phi_{N,k}(|z|)$.

\section{\texorpdfstring{Dynamics on the eigenspaces of $L$}{Dynamics on the eigenspaces of L}} 

\label{sectioneigenL}

Adopting radial coordinates $(r,\theta)$, let us set
$$
F_n = \big\{ e^{in\theta} f(r), \; \mbox{with} \; f \in L^2(\mathbb{R}^2) \big\}.
$$
It is the $n$-th eigenspace of the rotation operator $L$, which is left invariant by the dynamics of  \eqref{CR}. An eminent instance is the set of radial functions $F_0$.

\subsection{\texorpdfstring{Dynamics on $F_0$, the set of radial functions}{Dynamics on F0, the set of radial functions}} 

The following can be found in   \cite[Chapter 1]{Thanga} and  \cite[Corollary 3.4.1]{Thanga}: the Laguerre polynomial $L^{(0)}_k$ of type $0$ and degree $k\geq 0$ is defined by 
 \begin{equation}\label{laguerre}
 \e^{-x}L^{(0)}_k(x)=\frac1{k !}\frac{d^k}{dx^k}\big( \e^{-x}x^{k}  \big),\quad x\in \R.
 \end{equation}
 These polynomials are orthonormal on $L^{2}([0,+\infty), \e^{-x}dx)$
 \begin{equation}\label{ortho}
 \int_0^{+\infty}L^{(0)}_k(x)L^{(0)}_j(x) \e^{-x} dx= \delta_{jk},
 \end{equation}
and are related to special Hermite functions of second index 0 by 
\begin{equation*}
\phi_{2k,0}(x)=\frac1{\sqrt{\pi}}L^{(0)}_{k}(\vert x \vert^2)\e^{-\vert x\vert^2/2},
\end{equation*}
Simply denote $h_k = \phi_{2k,0}$, then $H h_k= (4k+2) h_k$, 
\begin{equation*}
\mathcal E(h_{n_1},h_{n_2},h_{n_3},h_{n_4})= \pi^2 \big(\int h_{n_1}h_{n_2}h_{n_3}h_{n_4}\big)\mathbbm{1}_{n_1+n_2=n_3+n_4}
\end{equation*}
and
\begin{equation*}
\T(h_{n_1},h_{n_2},h_{n_3})= \pi^2 \big(\int h_{n_1}h_{n_2}h_{n_3}h_{n_4}\big)h_{n_4}, \quad  n_4=n_1+n_2-n_3.
\end{equation*}

Write $\dis f=\sum_{n=0}^{+\infty}c_nh_n$. Then by \eqref{eqmod} the equation  \eqref{CR} is equivalent to 
\begin{equation*}
i\dot{c}_n= \pi^2 \sum_{\substack {n_1,n_2,n_3\geq 0\\n_1+n_2-n_3=n}}c_{n_1}c_{n_2}\ov{c_{n_3}}\big(\int h_{n_1}h_{n_2}h_{n_3}h_{n_4}\big).
\end{equation*}

It was already established in~\cite{FGH} that centered Gaussians generate stationary waves in $F_0$. Another stationary solution exhibited there is the self-similar function $\frac{1}{r}$, which does not belong to $L^2$ but is in the generalized $0$-eigenspace of $L$. Finally, all the $h_n$ give rise to stationary waves: $u(t) = e^{-i \omega_n t} h_n$, with $\omega_n = \mathcal{E}(h_n)$. 

Other stationary waves can be obtained by letting the symmetries of the system act on them. The symmetries of  \eqref{CR} which leave the set of radial functions invariant are $u \mapsto e^{i\theta} u$, $u \mapsto e^{i \mu |x|^2} u$, $u \mapsto e^{i \nu H} u$, $u \mapsto S_\lambda u = \lambda u (\lambda \cdot)$ and $u \mapsto e^{i \alpha \Delta} u$. Since the lens transform formula~(\ref{lenstransf}) expresses the fifth symmetry in terms of the four first symmetries, it suffices to consider the four first ones. Applying them to $h_{n}$ gives the orbit
$$
\mathcal{O}_n = \{ e^{i\theta} e^{i \nu H} e^{i \mu |x|^2} S_\lambda h_n, (\theta, \nu, \mu, \lambda) \in \mathbb{R}^4 \}.
$$
In the case of the Gaussian, we obtain the orbit
$$
\mathcal{O}_0 =  \{ e^{i\theta} e^{i \nu H} e^{i \mu |x|^2} S_\lambda h_0, (\theta, \nu, \mu, \lambda) \in \mathbb{R}^4 \}.
$$
It was proved in~\cite{FGH} that the Gaussian is orbitally stable in $L^2$ in the sense that data close to $\frac{1}{\sqrt{\pi}}e^{-\frac{1}{2}|x|^2}$ in $L^2$ yield solutions remaining close to $\mathcal{O}_0$ for all later times.

For this reason, it is interesting to express the orbit $\mathcal{O}_0$ in the $(c_n)$ coordinates:
\begin{align*}
\langle  e^{i\theta} e^{i \nu H} e^{i \mu |x|^2} S_\lambda h_0 \,,\, h_n \rangle = e^{i\theta} e^{i\nu(4n+2)} \langle e^{i \mu |x|^2} S_\lambda h_0 \,,\, h_n \rangle.
\end{align*}
We now successively change variables to $z = |x|^2$, use the formula~(\ref{laguerre}) giving $h_n$, and integrate by parts repetitively to obtain
\begin{align*}
\langle e^{i \mu |x|^2} S_\lambda h_0 \,,\, h_n \rangle & = \frac{\lambda}{\pi} \int_{\R^{2}} e^{i\mu|x|^2} e^{- \frac{\lambda^2}{2} |x|^2} {L}^{(0)}_n(|x|^2)e^{-|x|^{2}/2}\,dx \\
& = \frac{\lambda}{n!} \int_{0}^{+\infty} e^{i \mu z} e^{\frac{(1-\lambda^2)}{2}z} \left( \frac{d}{dz} \right)^n (z^n e^{-z}) \,dz \\
& =  \frac{(-1)^{n}\lambda}{n!} \left( \frac{1 - \lambda^2}{2} + i\mu\right)^n \int_0^{+\infty} e^{(i\mu - \frac{1+\lambda^2}{2})z} z^n \,dz\\
& = -\lambda \frac{\left(\frac{1-\lambda^2}{2} + i\mu\right)^n} {\left(-\frac{1+\lambda^2}{2} + i\mu\right)^{n+1}} .
\end{align*}
Therefore, $\langle  e^{i\theta} e^{i \nu H} e^{i \mu |x|^2} S_\lambda h_0 \,,\, h_n \rangle = -e^{i\theta} e^{i\nu(4n+2)} \lambda\frac{\left(\frac{1-\lambda^2}{2} + i\mu\right)^n} {\left(-\frac{1+\lambda^2}{2} + i\mu\right)^{n+1}}$, and the orbit $\mathcal{O}_0$ reads, in the $(c_n)$ coordinates
$$
\mathcal{O}_0 = \Big\{ c_n = - e^{i\theta} e^{i\nu(4n+2)} \lambda\frac{\left(\frac{1-\lambda^2}{2} + i\mu\right)^n} {\left(-\frac{1+\lambda^2}{2} + i\mu\right)^{n+1}}, \; (\theta, \nu, \mu, \lambda) \in \mathbb{R}^4 \Big\} .
$$

\subsection{\texorpdfstring{Dynamics on $F_n$, $n\neq 0$}{Dynamics on Fn}}

The subspace $F_n$ does not contain Gaussians anymore, and the role of the ground state is played by
$$
\phi_{n,n}(z) = \frac{1}{\sqrt{\pi n!}} z^n e^{-\frac{|z|^2}{2}}.
$$
It is not clear whether it minimizes $\mathcal{E}$ on $F_n$ for fixed mass; but it is clearly the minimizer of $\langle E \phi,\phi \rangle$ on $F_n$ for fixed mass. This gives immediately orbital stability on $F_n$ for the $M$-stationary wave arising from $\phi_{n,n}$. We study the dynamics on this family $\{\phi_{n, n}\}$ in the next section.

\subsection{Other equivariant stationary waves} We first claim that the equality
$$
\mathcal{T}\Big( \frac{e^{in\theta}}{r},  \frac{e^{in\theta}}{r}, \frac{e^{in\theta}}{r} \Big) = \omega  \frac{e^{in\theta}}{r}
$$
holds for some real number $\omega$ (recall that $\mathcal{T}$ is bounded on $\dot L^{\infty,1} = \{ f \; \mbox{such that} \; rf \in L^\infty\}$, thus the left-hand side makes perfect sense).  Recall that $\mathcal{T}$ commutes with the rotation operator $L$ and the dilation operator $S_\lambda$. Therefore, $\mathcal T \big( \frac{e^{in\theta}}{r}, \frac{ e^{in\theta}}{r},\frac{ e^{in\theta}}{r} \big)$ must be invariant by $S_\lambda$, and of the form $e^{in\theta}f(r)$; thus it has to be equal to $\omega  \frac{e^{in\theta}}{r}$ for some $\omega$.

As a consequence, we obtain a new stationary wave in the generalized $n$-th eigenspace of $L$: $e^{-i\omega t}  \frac{e^{in\theta}}{r}$.

\section{Dynamics on the Bargmann-Fock space} 

\label{sectionbargmannfock} 

Denote by $\mathcal{O}(\C)$ the space of the entire functions in the complex plane. Then the Bargmann-Fock space is given by $L^{2}(\R^2)\cap (\mathcal{O}(\C) \e^{-\vert z \vert^2/2})$.
It admits an orthonormal basis given by the special Hermite functions $\phi_{n,n}$, which we will simply denote $\phi_n$ to alleviate notations:
\begin{equation}\label{def1}
\phi_{n}(x_1,x_2)=\frac{1}{\sqrt{\pi n!}}(x_1+i x_2)^n \e^{-\vert x\vert^2/2 }.
\end{equation}
Recall that $\phi_n$ is such that $H \phi_n=2(n+1)\phi_n$ and $\Vert \phi_n\Vert_{L^2(\R^2)}=1.$
 
It is an invariant subspace for  \eqref{CR}, and we consider in this section its dynamics restricted to it. Out of all the symmetries of  \eqref{CR}, only three act on its restriction to the Bargmann-Fock space: phase rotation $u \mapsto u e^{i\theta}$, with $\theta \in \mathbb{R}$; space rotation $u \mapsto R_\theta u$, with $\theta \in \mathbb{R}$; and magnetic translations $u \mapsto u(z + \xi) e^{-\frac{1}{2}(\bar \xi z - \xi \bar z)}$, for $\xi \in \C$. These symmetries are associated by Noether's theorem to the three conserved quantities
$$
M = \int_{\C} |u(z)|^2 \,dz,\quad P= \int_{\mathbb{C}} (\vert z\vert^2-1)\vert u(z)\vert^2 \,dz, \quad  Q= \int_{\mathbb{C}}  z\vert u(z)\vert^2 dz.
$$

\subsection{\texorpdfstring{The $L^2$ framework}{The L2 framework}}

\begin{lemm}\ph \label{lem6}
Let $\phi_n$ be defined by \eqref{def1}. Then 
\begin{equation*} 
\T(\phi_{n_1},\phi_{n_2},\phi_{n_3}) = \alpha_{n_1,n_2,n_3,n_4} \phi_{n_4}, \qquad n_4=n_1+n_2-n_3,
\end{equation*} 
with
\begin{equation} \label{defalpha}
\alpha_{n_1,n_2,n_3,n_4} = \mathcal E(\phi_{n_1},\phi_{n_2},\phi_{n_3},\phi_{n_4}) =\frac{\pi }{2} \frac{(n_1+n_2)!}{2^{n_1+n_2}\sqrt{n_1 !n_2 !n_3 !n_4 !}}\mathbbm{1}_{n_1+n_2=n_3+n_4}.
\end{equation} 
\end{lemm}
  
\begin{proof} The first claim follows from Proposition \ref{propFormule}

With the change of coordinates $z=r \e^{i\theta}$ and $\rho= 2r^2$, we get
\begin{eqnarray*}
\mathcal E(\phi_{n_1},\phi_{n_2},\phi_{n_3},\phi_{n_4})&=& {\pi^2} \int_{\R^2} \phi_{n_1}\phi_{n_2}\ov{\phi_{n_3}}\ov{\phi_{n_4}}  \\
&=& \frac{1}{ \sqrt{n_1! n_2! n_3! n_4!}}\int_{\R^2} (x_1+i x_2)^{n_1+n_2}(x_1-i x_2)^{n_3+n_4} e^{-2|x|^2} dx\\
&=& \frac{1}{ \sqrt{n_1! n_2! n_3! n_4!}}\int_{-\pi}^{\pi }  \e^{i(n_1+n_2-n_3-n_4)\theta} d\theta\int_{0}^{+\infty} r^{n_1+n_2+n_3+n_4+1}\e^{-2r^2} dr\\
&=& \frac{\pi}{2\cdot 2^{n_1+n_2} \sqrt{n_1! n_2! n_3! n_4!}}\Big(\int_{0}^{+\infty} \rho^{n_1+n_2}\e^{-\rho} d\rho\Big) \mathbbm{1}_{n_1+n_2=n_3+n_4}\\
&=& \frac{\pi(n_1+n_2)!}{2 \cdot 2^{n_1+n_2} \sqrt{n_1! n_2! n_3! n_4!}}  \mathbbm{1}_{n_1+n_2=n_3+n_4},
\end{eqnarray*}
which was the second claim.
  \end{proof}
  
As a consequence we a have the following result.     

\begin{lemm} Denote by $\Pi$ the orthogonal projector on the space $\mathcal{O}(\C) \e^{-\vert z \vert^2/2}$. Then
\begin{equation*}
\big(\Pi u\big) (z)= \frac{1}{\pi} e^{-\frac{|z|^2}{2}} \int_\mathbb{C} e^{\ov{w}  z - \frac{|w|^2}{2}} u(w) \,dw,
\end{equation*}
and we have
 \begin{equation*}
\T(\phi_{n_1},\phi_{n_2},\phi_{n_3} )= {\pi^2} \Pi\big(\phi_{n_1}\phi_{n_2}\ov{\phi_{n_3}}\big)=
 {\pi} \Big(\int_{\xi\in \C} \phi_{n_1}(\xi)\phi_{n_2}(\xi)\ov{\phi_{n_3}}(\xi)e^{\ov{\xi}z}e^{-\vert \xi\vert^2/2}d\xi\Big) e^{-\vert z\vert^2/2}.
\end{equation*}
\end{lemm} 

\begin{proof}
The first point  follows from the fact that the kernel $K$ of $\Pi$ is given by 
\begin{equation*}
K(z,\xi)=\sum_{n=0}^{+\infty}\phi_n(z)\ov{\phi_n}(\xi)=\frac{1}{\pi}e^{\ov{\xi}z}e^{-\vert \xi\vert^2/2}e^{-\vert z\vert^2/2}.
\end{equation*}
The reformulation of $\T$ in terms of $\Pi$ is then a direct computation using a polar change of variables.
\end{proof}

Hence the \eqref{CR} equation reads on $L^{2}(\R^2)\cap (\mathcal{O}(\C) \e^{-\vert z \vert^2/2})$
\begin{equation*}
i\partial_t u= {\pi} \Pi \big( |u|^2 u\big).
\end{equation*}
We remark the resemblance of this equation the Szeg\"o equation of Gerard and Grellier where there $\Pi$ is the Szeg\"o projector \cite{GG1}. 
\subsection{\texorpdfstring{Stability in $L^2$ of stationary waves}{Stability in L2 of stationary waves}}

 Consider a Gaussian solitary wave in the Bargmann-Fock space: $u(z) = \frac{1}{\sqrt{\pi}} e^{-\frac{|z|^2}{2}} e^{-i \omega t}$. Its  stability   follows directly from the results in~\cite{FGH}.
\begin{itemize}
\item It is orbitally stable with respect to perturbations in $L^{2,1} \cap H^1$ (see \cite[Proposition 6.8]{FGH}).
\item With respect to perturbations in $L^2$, it is orbitally stable modulo the symmetries acting on the system (phase rotation, space rotation, and magnetic translations). In other words, any perturbation of it remains close to $\{ \frac{1}{\sqrt{\pi}} e^{i\theta -\frac{1}{2} |\xi|^2 -\frac{1}{2}|z|^2 - z \bar \xi}, \; \mbox{with} \; \theta \in \mathbb{R} \; \mbox{and} \; \xi \in \mathbb{C} \}$ (see \cite[Proposition 8.5]{FGH}).
\end{itemize}
\medskip

We now would investigate the stability of the stationary waves $\phi_N e^{-i \omega t}$, with $\omega_N = \mathcal{E}(\phi_N)$, for ${N \geq 1}$. Since nonlinear stability seems to be a delicate question, we focus on linear stability (understood as the absence of exponentially growing mode). 

\begin{prop} For $N\geq 0$, consider the wave  $\phi_N e^{-i \omega_N t}$. The following stability/instability results hold in the Bargmann-Fock space:
\begin{enumerate}[(i)]
\item If $N=0$ or $N=1$, this wave is linearly stable.
\item If $N\geq2$, it is linearly unstable.
\item The number of unstable modes of the $N$-th wave is $o(N)$.
\end{enumerate}
\end{prop}

\begin{proof} We already have the result for the Gaussian, thus we can assume that $N\geq 1$. It will be convenient to use the basis provided by the special Hermite functions $\phi_n$: write $u = \sum_{n=0}^\infty c_n \phi_n$. The stationary waves whose linear stability we will investigate read in these coordinates
\begin{equation*}
c_n = \delta_{N,n}
\end{equation*}
(where $\delta$ is the Kronecker delta function). The linearization of  \eqref{CR} restricted to the Bargmann-Fock space around this stationary wave is given by
\begin{align}
\label{martinpecheur1} & i\dot{c_N} = 2 \alpha_{NNNN} c_N +  \alpha_{NNNN} e^{-2i\alpha_{NNNN}t}\, \overline{c_N} \\
\label{martinpecheur2}& i\dot{c_{k}} = 2 \alpha_{kNkN} c_{k}  + \alpha_{2N-k,k,N,N}  e^{-2i \alpha_{NNNN} t}  \,\overline{c_{2N-k}} \quad \mbox{if $k \in \{0 \dots N-1\} \cup \{N+1 \dots 2N\}$} \\
\label{martinpecheur3}& i\dot{c_{k}} = 2 \alpha_{kNkN} c_{k}  \quad \mbox{if $k \geq 2N+1$},
\end{align}
where $\alpha$ is given by \eqref{defalpha}. The equation \eqref{martinpecheur3} is obviously stable. As for the equation~\eqref{martinpecheur1}, the change of unknown variable $c_N = e^{-i\alpha_{NNNN} t} z$ leads to the equation
$$
i \dot z = \alpha_{NNNN}(z+\overline{z}),
$$
whose solutions grow at most linearly. Finally, to study the equation \eqref{martinpecheur2}, observe that it only couples $c_k$ and $c_{2N-k}$. Set $k \in \{0 \dots N-1\} \cup \{N+1 \dots 2N\}$ and write
$$
\omega = \alpha_{NNNN},\;\; c_{k} = e^{-i\omega t} x, \;\; c_{2N-k} = e^{-i\omega t} y, \;\;A = \alpha_{k,2N-k,N,N},\;\; B = \alpha_{kNkN}, \;\; C = \alpha_{2N-k,N,2N-k,N}.
$$
The equation \eqref{martinpecheur2} becomes
\begin{equation*} 
\left\{
\begin{aligned}
& i \dot x = A \overline{y} + (2B-\omega) x \\
& i \dot{y} = A \overline{x} + (2C - \omega) y, 
\end{aligned}
\right.
\end{equation*}
which implies
$$
\ddot x + i(2B-2C) \dot x - \big(A^2 - (2C - \omega)(2B-\omega)\big) x = 0.
$$
This equation has exponentially growing modes if and only if its discriminant is positive:
$$
\Delta(N,2N-k) =\Delta(N,k) = 4 \big(A^2 - (B+C-\omega)^2\big) >0,
$$
for $0\leq k\leq N-1$. The results of the proposition are then implied by the following lemma.
\end{proof}

\begin{lemm} With $\Delta(N,k)$ given by the above definition, 
\begin{enumerate}[(i)]
\item $\Delta(1,0)<0$.
\item $\Delta(N,N-2)>0$ for all $N \geq 2$.
\item For any $\lambda \in (0,1)$,  $\Delta(N,k) < 0$ for $k =\lambda N$ and $N$ sufficiently big.
\end{enumerate}
\end{lemm}

\begin{proof} The part $(i)$ is straightforward. We then check $(ii)$. Thanks to the formula giving $\alpha$, we find for $k=N-2$
\begin{align*}
A+B+C-\omega & = \alpha_{N-2,N+2,N,N} + \alpha_{N-2,N,N-2,N} + \alpha_{N+2,N,N+2,N} - \alpha_{N,N,N,N} \\
& =  \frac{\pi(2N-2)!}{(N-1)! (N-2)! 2^{2N+1}} \Big(\frac{2(2N-1)}{\sqrt{(N+2)(N+1)N(N-1)}}   +\frac{(4N^2-2N-5)(N+1)}{(N+2)(N+1)N(N-1)} \Big).
\end{align*}
Since $4N^2-2N-5 > 0$ for $N \geq 2$, we find that $A+B+C - \omega >0$ for all $N\geq 2$ Next,
\begin{align*}
A-B-C+\omega & = \alpha_{N-2,N+2,N,N} - \alpha_{N-2,N,N-2,N} - \alpha_{N+2,N,N+2,N} - \alpha_{N,N,N,N} \\
& =  \frac{\pi(2N-2)!}{(N-1)! (N-2)! 2^{2N+1}} \Big(\frac{2(2N-1)}{\sqrt{(N+2)(N+1)N(N-1)}}   -\frac{(4N^2-2N-5)(N+1)}{(N+2)(N+1)N(N-1)} \Big).
\end{align*}
This term is positive iff
\begin{equation*}
4(2N-1)^2(N+2)N(N-1)>(4N^2-2N-5)^2(N+1),
\end{equation*}
or equivalently $8N^3+52N^2-53N-25>0$, which is the case for all $N\geq 2$. Since  $\Delta = 4(A+B+C+\omega)(A-B-C-\omega)$, we find that  $\Delta(N,N-2) >0$ for all $N \geq 2$.

To prove $(iii)$, we use Lemma~\ref{lem6} and the Stirling formula to prove that, if $k = \lambda  N$ for some $\lambda \in (0,1)$, as $N \to \infty$,
\begin{align*}
& A \sim \frac{1}{2} \sqrt{\frac{\pi}{\sqrt{\lambda (2-\lambda)}}} \frac{1}{\sqrt{N}} \Big( \frac{1}{{\lambda^{\frac{\lambda}{2}} (2-\lambda)^{1-\frac{\lambda}{2}}}}  \Big)^N \\
& B \sim \frac{\sqrt{2\pi}}4  \sqrt{\frac{1+\lambda}{\lambda}} \frac{1}{\sqrt{N}}  \left( \frac{(1+\lambda)^{1+\lambda}}{\lambda^\lambda 2^{1+\lambda}}\right)^N \\
& C \sim  \frac{\sqrt{2\pi}}4 \sqrt{\frac{3-\lambda}{2-\lambda}}  \frac{1}{\sqrt{N}} \left( \frac{ (3-\lambda)^{3-\lambda}}{(2-\lambda)^{2-\lambda}2^{3-\lambda}}\right)^N \\
& \omega \sim \frac{\sqrt{\pi}}2 \frac{1}{\sqrt N}.
\end{align*}
Observe now that, if $\lambda \in (0,1)$, $ \frac{1}{{\lambda^{\frac{\lambda}{2}} (2-\lambda)^{1-\frac{\lambda}{2}}}}$, $ \frac{(1+\lambda)^{1+\lambda}}{\lambda^\lambda 2^{1+\lambda}}$, and $\frac{ (3-\lambda)^{3-\lambda}}{(2-\lambda)^{2-\lambda}2^{3-\lambda}}$ also belong to $(0,1)$. This implies that $A$, $B$, and $C$ decay exponentially in $N$ while $\omega$ only decays polynomially in $N$. Since $\Delta = 4 (A^2 - (B+C+\omega)^2)$, this implies in turn that, for $\lambda$ fixed, $\Delta(N,\lambda N)<0$ for $N$ sufficiently large.
\end{proof}

\subsection{\texorpdfstring{The $L^\infty$ framework}{The L infinity   framework}} 

It is easy to check that the orthogonal projector on $\mathcal{O}(\C) \e^{-\vert z \vert^2/2}$ is bounded on $L^\infty$, and thus can be naturally extended from $L^2$ to $L^\infty$. Therefore, the equation
$$
i\partial_t u= {\pi} \Pi\big(|u|^{2}u \big)
$$
is locally well posed in $L^\infty_t L^\infty_x$ for data in $L^\infty$.

Here we make the link with a result of Aftalion-Blanc-Nier \cite[Theorem 1.4]{ABN}. Denote by $\Pi_{hol}$ the orthogonal projection on the space of entire functions on $\C$ (in \cite{ABN} this corresponds to $\Pi_{h}$ with $h=1$). For $f$, set $u=e^{-|z|^{2}/2}f$, then 
\begin{equation*}
\Pi u=\Pi\big(e^{-|z|^{2}/2}f\big)=\frac1\pi\Big(\int_{\xi\in \C}  f(\xi)e^{\ov{\xi}z}e^{-\vert \xi\vert^2}d\xi\Big) e^{-\vert z\vert^2/2}=e^{-|z|^{2}/2}\Pi_{hol} f.
\end{equation*}
Thus 
\begin{equation*}
\Pi\big(|u|^{2}u \big)= e^{-|z|^{2}/2}\Pi_{hol} \big(e^{-|z|^{2}}|f|^{2}f\big).
\end{equation*}
As a consequence, the function $u_{\tau}(z)=e^{-|z|^{2}/2}f_{\tau}(z)$ given by \cite[Theorem 1.4]{ABN} is a stationary solution to \eqref{CR}. The function $f_{\tau}$ belongs to the space $L^\infty$.

\appendix

\section{Two general structures}

The various equations derived in the present paper present striking similarities, they all belong to one, or two, of the general structures described below. The presentation we give is only formal.

\subsection{An equation on sequences}

Assume that $(\alpha_{k l m n})_{(k,l,m,n) \in A^4}$, where $A \subset{\Z}$ satisfies the symmetries
$$
\alpha_{k l m n} = \alpha_{l k m n} \quad \mbox{and} \quad \alpha_{k l m n} = \overline{\alpha_{m n k l}}.
$$  
Then the  equation 
$$i \dot{c_n} = \sum_{n_1+n_2=n_3+n} \alpha_{n_1 n_2 n_3 n} c_{n_1} c_{n_2} \ov{c_{n_3}}$$
derives  from the Hamiltonian
$$
\mathcal{G} ((c_n)) = \frac{1}{4} \sum_{n_1+n_2=n_3+n_4} \alpha_{n_1 n_2 n_3 n_4} c_{n_1} c_{n_2} \overline{c_{n_3} c_{n}}.
$$
Conserved quantities for this equation are the $\ell^2$ norm as well as the Hamiltonian. 

Examples are
\begin{itemize}
\item $A = \mathbb{Z}$, $\alpha_{n_1 n_2 n_3 n_4} = 1$, which is simply the equation $i\dot u = |u|^2 u$ on $\mathbb{T}$ seen in Fourier space.
\item $A = \mathbb{N}$, $\alpha_{n_1 n_2 n_3 n_4} = \mathbbm{1}_{k,l,m,n \geq 0}$, which is the Szeg\"o equation.
\item $A = \{0, \dots, N\}$, $\alpha_{n_1 n_2 n_3 n_4} = \mathcal{E}(\phi_{N,n_1},\phi_{N,n_2},\phi_{N,n_3},\phi_{N,n_4})$, which is the equation on $E_N$.
\item $A = \mathbb{N}$, $\alpha_{n_1 n_2 n_3 n_4} = \frac{\pi}{8} \frac{(n_1+n_2)!}{2^{n_1+n_2} \sqrt{n_1! n_2 ! n_3! n_4!}}$ which is the Lowest-Landau-Level equation.
\item $A = \mathbb{N}$, $\alpha_{n_1 n_2 n_3 n_4} = \frac{\pi^2}{4} \int h_{n_1} h_{n_2} h_{n_3} h_{n_4}$ where the $(h_n)$ are the normalized radial Hermite functions, which is the "radial equation".
\end{itemize}

\subsection{An equation on functions of a continuous variable}

Assume that $E$ is a closed subset of $L^2(\mathbb{R}^k)$, with orthogonal complement $E^{\perp}$. Let $\Pi$ be the orthogonal projection on $E$.

The equation (whose dependent variable is a function $u$ on $\mathbb{R}^k$).
$$
i \dot{u} = \Pi \left( |u|^2 u \right)
$$
derives from the Hamiltonian
$$
\mathcal{E}(u) = \int |u|^4 \quad \mbox{on $E$}.
$$
Conserved quantities are $\int |u|^4$ and $\int |u|^2$.

Instances are
\begin{itemize}
\item $E = L^2(\mathbb{R}^k)$, which is simply the equation $\dot u = |u|^2 u$.
\item $k=1$, $E = L^2_+$ (functions with positive frequencies), which is the Szeg\"o equation.
\item $k=2$, $E = E_N$, which is the equation on $E_N$.
\item $k=2$, $E = e^{-|x|^2/2} \mathcal{O}(\C)$, which is the Lowest-Landau-Level equation.
\end{itemize}

\bigskip

{\bf Acknowledgement:} The authors are grateful to Patrick G\'erard for pointing out the connection with the Lowest-Landau-Level equation. This work was initiated during the visit of the third author to the Courant Institute of Mathematical Sciences, and he thanks the Institute for its hospitality.

\end{document}